%% file: main.tex
\documentclass{article}
\usepackage[tmargin=1in,bmargin=1in,lmargin=1.2in,rmargin=1.2in]{geometry}
\usepackage{graphicx} 
\usepackage{tikz}
\usepackage{amsmath,amsthm,amssymb,amsfonts,geometry,mathtools,float}
\usepackage{amsfonts, amssymb, amscd, mathrsfs}
\usepackage{hyperref}
\usepackage{geometry}
\usepackage{comment}
\usepackage{url}
\usepackage{appendix}
\newtheorem{theorem}{Theorem}
\newtheorem{lemma}[theorem]{Lemma}

\newtheorem{corollary}[theorem]{Corollary}
\newtheorem{question}[theorem]{Question}
\theoremstyle{definition}

\newtheorem{remark}[theorem]{Remark}

\usepackage{tikz}

\usepackage{xcolor}

\newcommand{\F}{{\mathbb F}}

\newcommand{\Q}{{\mathbb Q}}
\newcommand{\R}{{\mathbb R}}
\newcommand{\Z}{{\mathbb Z}}
\newcommand{\Frob}{\text{Frob}}

\definecolor{cite}{rgb}{0.30,0.60,1.00}
\newcommand{\diana}[1]{{\color{blue} \sf $\clubsuit\clubsuit\clubsuit$ Diana: [#1]}}

\title{Weil polynomials of abelian varieties over finite fields}
\author{Michael Cerchia, Zeyu Liu, Diana Mocanu, Haodong Yao, Jing Ye}
\date{}

\begin{document}
\include{macro}
\maketitle
\begin{abstract}
    In this paper, we investigate Weil polynomials and their relationship with isogeny classes of abelian varieties over finite fields. 
    We give a necessary condition for a degree $12$ polynomial with integer coefficients to be a Weil polynomial. 
     Moreover, we provide explicit criteria that determine when a Weil polynomial of degree $14$ occurs as the characteristic polynomial of a Frobenius endomorphism acting on an abelian variety. 
\end{abstract}

\section{Introduction}\label{intro}
Let $A$ be an abelian variety of dimension $g$ over a finite field $\FF_q$, where $q = p^n$ is a prime power. Let $T_\ell(A)=\varprojlim A[\ell^n]$ be the $\ell$-adic Tate module of $A$ and $V_\ell(A) = T_\ell(A)\otimes_{\ZZ_\ell}\QQ_\ell$. For $\ell\neq p$, the characteristic polynomial of the Frobenius endomorphism $\operatorname{Frob}_A$ of $A$ is defined to be $$\chi_A(t) = \det (\operatorname{Frob}_A-t I\mid V_\ell(A)),$$ which is a monic polynomial of degree $2g$ with integer coefficients that is independent of the choice of the prime $\ell$. Moreover, $\chi_A(t)$ can be written as $$\chi_A(t) = t^{2g}+a_1t^{2g-1}+\cdots+a_gt^g+qa_{g-1}t^{g-1}+\cdots+q^{g-1}a_1t+q^g,$$ 
with $a_i\in\Z$, and its set of roots in $\conj{\mathbb{Q}}$ has the form $\{\omega_1,\cdots,\omega_g,\conj{\omega}_1,\cdots,\conj{\omega}_g\}$,
where $\omega_i$
is a $q$-\textit{Weil number} for $i\in \{1, \dots g\}$.
We recall that a $q$-\textit{Weil number} $\omega$ is an algebraic integer such that $|\sigma(\omega)| = \sqrt{q}$ for any embedding $\sigma: \QQ(\omega)\to \CC$, and a $q$-\textit{Weil polynomial} is a monic polynomial with integer coefficients whose roots are $q$-Weil numbers coming in pairs $\{\omega_i,\conj{\omega}_i\}_{i= 1}^g$. We will often omit the $q$- when the context is clear. In particular, the characteristic polynomial $\chi_A(t)$ of Frobenius is a Weil polynomial. 

A standard result due to Tate asserts that the characteristic polynomial is an isogeny invariant. More precisely,
given two abelian varieties $A,B$ defined over $\F_q$, Theorem $1$ in \cite{tate} states that $A$ is $\F_q$-isogenous to $B$ if and only if $\chi_A(t)=\chi_B(t)$. Further, any abelian variety can be decomposed up to isogeny into simple abelian varieties (i.e., those with no non-zero proper abelian subvarieties) and the characteristic polynomial is compatible with this decomposition. To see this, let $A$ be an abelian variety defined over $\F_q$. Then $A$ is isogenous to a product 
\[
A \sim A_1^{r_1} \times \ldots \times A_m^{r_m},
\]
where each $A_i$ is a simple abelian variety over $\F_q$ such that $A_i \not\sim A_j$ for $i\neq j $, and $r_i \geq 1$ is a positive integer. If $\chi_{A_i}$ is the characteristic polynomial of the Frobenius endomorphism of $A_i$, we have that 
\[
\chi_A(t)=\chi_{A_1}(t)^{r_1}\ldots\chi_{A_m}(t)^{r_m}.
\]
Therefore understanding $\chi_A(t)$ for abelian varieties $A$ over finite fields of dimension $g$ reduces to understanding $\chi_A(t)$ of simple abelian varieties $A$ of $\dim(A)\leq g$. An essential property of $\chi_A(t)$ for $\F_q$-simple abelian varieties $A$ over $\F_q$ is that 
\[
\chi_A(t)=m_A(t)^e
\]
where $m_A(t)$ is an irreducible polynomial and $e\geq 1$ an integer. We call $e$ \textit{the multiplicity} of $A$ and we note that $e\mid 2\dim(A)$.
In this paper, we study the following two questions.
\begin{question}\label{Q1}
    Under what conditions is a monic polynomial of degree $2g$ with integer coefficients a Weil polynomial?
\end{question}
In Section \ref{sec:Q1}, we partially address this question for $g=6$ by showing the size of the coefficients of a given Weil polynomial must respect certain constraints.
 For dimensions $g=1$ and $g=2$ see Rück \cite{Rück1990} and \cite{Waterhouse1969}, respectively. Haloui \cite{Hal2010} addressed dimension $3$, while Haloui and Singh \cite{HS2012} tackled dimension $4$. Progress on the general problem so far has been summarised by Hayashida \cite{hay2}. The case $g=5$ was done by Sohn \cite{sohn5}.
 The first steps of $g=3,4$ and $5$ contained some mistakes, which were later corrected by Lin in his master's thesis \cite{Lin}. In a recent preprint, Marseglia polished these results in \cite{marseglia2025}.
 Next, we answer: 
\begin{question}\label{Q2}
    Under what conditions does a given Weil polynomial of degree $2g$ occur as the
characteristic polynomial of Frobenius for a simple abelian variety of dimension $g$ over a given finite field?
\end{question}

In Section \ref{sec:Q2}, we address this question for dimension \( g = 7 \). The solution for $g=1$ was established in \cite{Waterhouse1969}. For dimension $g=2$, the results can be found in \cite{Rück1990} and \cite{MNH2002}, while dimensions $g=3$ and $g=4$ are discussed in \cite{Hal2010}, \cite{xing}, and \cite{HS2012}, respectively. The case for $g=5$ is covered in \cite{hay1}, and $g=6$ in \cite{sohn2020}.
We answer Question~\ref{Q2} by giving sufficient and necessary conditions on the $p$-adic valuations on the coefficients of the given Weil polynomial. Our method involves classifying Newton Polygons and could be generalized in theory to any prime dimension $g$.  

\section{Acknowledgements} We thank the organizers of the 2024 Arizona Winter School on Abelian Varieties, where this project originated. In particular, we thank Valentijn Karemaker for suggesting this project and for carefully reading our draft and making many helpful suggestions. We also thank Soumya Sankar for her mentorship and Kiran Kedlaya for giving us an unlimited CoCalc account, where many computations for this project were carried out. We are grateful to Stefano Marseglia for his useful comments and suggestions.
During the preparation of the project, the second author was partially supported by NSF DMS-2053473 under Professor Kedlaya. The third author was involved in writing this article while transitioning from the University of Warwick to the Max Planck Institute for Mathematics in Bonn, and thus she is thankful to both institutions for their kind support.

\section{Question \ref{Q1} for \texorpdfstring{$\boldsymbol{g=6}$}{g=6}}\label{sec:Q1}
Let $\chi(t)\in \Z[t]$ be a degree $12$ polynomial of the form 
\begin{equation}\label{chi1}
	\chi(t) = t^{12}+a_1t^{11}+a_2t^{10}+\cdots +a_6t^6+qs_{5}t^{5}+\cdots+q^{5}a_1t+q^6.
\end{equation}
In this section, we give explicit bounds on the coefficients $a_i\in \ZZ$ when $\chi(t)$ is a Weil polynomial.

We firstly transform the conditions for $\chi(t)$ being a Weil polynomial into conditions for two associated real-coefficient polynomials $f$ and $\tilde{f}$ of degree $6$ having only real positive roots. By Rolle's theorem, if $f(t)$ has all roots positive and real, so does its derivative $f'(t)$. 

The work of Lin \cite[Lemma 10.1]{Lin}
provides necessary and sufficient conditions for a degree $5$ polynomial to have all roots positive real, and we apply this result to $f'(t)$ and $\tilde{f}'(t)$. 
He derives this result while studying Question~\ref{Q2} for $g=5$ using the so-called  Robinson’s method; see \cite[Chapter 7]{Lin} for more details.
In his case, the derivative is of degree $4$ and explicit radical expressions for the roots can be found. However, in our case, $f'(t)$ is of degree $5$, and it is well-known from standard Galois theory that in general there are no radical solutions for degree $5$ polynomials. Consequently, this method only gives conditions for $a_1,\dots,a_5$. To complete the resolution, we use the trivial Weil bound for $a_6$; see Remark \ref{rmk:bounds}.

\subsection{The case of real roots.}
Suppose first that $\chi(t)$ is a Weil polynomial with real roots. We will show that this case can be reduced to cases of lower-degree polynomials, which are already covered in the literature. Thus, we omit the details.

We start by recalling the fact that Weil polynomials have 
even multiplicity at real roots; see for example, Corollary $4.8$ in \cite{Lin} whose proof is valid for any Weil polynomial. If $\chi(t)$ has a real root and $\sqrt{q}$ is a square in the integers, then $\chi(t)$ is a Weil polynomial if and only if
\[
\chi(t)=(t+\sqrt{q})^2\widetilde{\chi}(t)\quad\text{ or }\quad \chi(t)=(t-\sqrt{q})^2\widetilde{\chi}(t),
\]
where $\widetilde{\chi}(t)$ is a Weil polynomial of degree $10$. 

In the case where $q$ is not a square, we know that the minimal polynomial of $\pm\sqrt{q}$ is $t^2-q$. So if $\sqrt{q}$ or $-\sqrt{q}$ is a root of $\chi(t)$ then $t^2-q$ must divide $\chi(t)$, and hence
$$\chi(t)=(t^2-q)^2\tilde{\chi}(t),$$
where $\widetilde{\chi}(t)$ is a Weil polynomial of degree $10$ once again.

We refer to \cite[Main Theorem C]{marseglia2025} for a complete resolution of Weil polynomials $\widetilde{\chi}(t)$ of degree $10$.
\subsection{The case of no real roots.}
Suppose now that $\chi(t)$ is a Weil polynomial with no real roots.
Therefore, its set of roots consists of complex numbers of the form $\{\omega_1,\bar{\omega}_1,\cdots,\omega_6,\bar{\omega}_6\}$ and with absolute value $\sqrt{q}$. For $1\leq i\leq 6$, we let $x_i = -(\omega_i+\bar{\omega}_i)$. Then $\chi(t)$ is given by 
\begin{equation}\label{chi2}
    \chi(t) = \prod_{i=1}^{6}(t^2+x_it+q).
\end{equation}
The first key observation (see \cite[Proposition 6.1.]{Lin}) is that an integer polynomial $\chi(t)$ of the form given in \eqref{chi1} is a Weil polynomial if and only if the polynomials $f(t) = \prod_{i=1}^6(t - (2\sqrt{q}+x_i))$ and $\tilde{f}(t) = \prod_{i=1}^6(t-(2\sqrt{q}-x_i))$ have only real positive roots.

For $1\leq i\leq 6$, let $v_i$ denote the $i$-th symmetric function of $x_i$'s. By comparing coefficients of \eqref{chi1} and \eqref{chi2}, one gets that

$$\begin{aligned}
	&v_1 = a_1,\\
	&v_2 = a_2- 6q,\\
	&v_3 = a_3 - 5qa_1,\\
	&v_4 = a_4 - 4qa_2+9q^2,\\
	& v_5 =  a_5 - 3qa_3 + 5q^2a_1,\\
	& v_6 =  a_6 - 2qa_4 +2q^2a_2-2q^3.
\end{aligned}$$
Let $r_i$ and $\tilde{r}_i$ be the $i$-th coefficients of $f(t)$ and $\tilde{f}(t)$:  
\begin{equation}\label{eq:f}
f(t) = t^6+r_1t^5+r_2t^4+r_3t^3+r_4t^2+r_5t+r_6,
\end{equation}
and 
\begin{equation}\label{eq:tildef}
    \tilde{f}(t) = t^6+\tilde{r}_1t^5+ \tilde{r}_2t^4+\tilde{r}_3t^3+\tilde{r}_4t^2+\tilde{r}_5t+\tilde{r}_6.
\end{equation}
Then we have 
$$\begin{aligned}
	&r_1 = -12\sqrt{q}-a_1,\\
	&r_2 =  54q+10\sqrt{q}a_1+a_2,\\
	&r_3 =  -112q\sqrt{q}-35qa_1-8\sqrt{q}a_2-a_3,\\
	&r_4 = 105q^2+50q\sqrt{q}a_1+20qa_2+6\sqrt{q}a_3+a_4, \\
	& r_5 =  -36q^2\sqrt{q}-25q^2a_1-16q\sqrt{q}a_2-9qa_3-4\sqrt{q}a_4-a_5, \\
	& r_6 =  2q^3+2q^2\sqrt{q}a_1+2q^2a_2+2q\sqrt{q}a_3+2qa_4+2\sqrt{q}a_5+a_6,
\end{aligned}$$
 and $$\begin{aligned}
 	&\tilde{r}_1 = -12\sqrt{q}+a_1,\\
 	&\tilde{r}_2 =  54q-10\sqrt{q}a_1+a_2,\\
 	&\tilde{r}_3 =  -112q\sqrt{q}+35qa_1-8\sqrt{q}a_2+a_3,\\
 	&\tilde{r}_4 = 105q^2-50q\sqrt{q}a_1+20qa_2-6\sqrt{q}a_3+a_4, \\
 	& \tilde{r}_5 =  -36q^2\sqrt{q}+25q^2a_1-16q\sqrt{q}a_2+9qa_3-4\sqrt{q}a_4+a_5, \\
 	& \tilde{r}_6 =  2q^3-2q^2\sqrt{q}a_1+2q^2a_2-2q\sqrt{q}a_3+2qa_4-2\sqrt{q}a_5+a_6.
 \end{aligned}$$
The next lemma gives necessary conditions for an integer polynomial $f$ of degree $6$ to have only positive real roots. More precisely, it gives bounds on the first $5$ coefficients of $f$. We prove it by using the analogous result for degree $5$ given in \cite[Corollary 3.6]{marseglia2025} on $f'$.
 
 \begin{lemma}\label{lem:ineq}
    Let $f(t) = t^6+r_1t^5+r_2t^4+r_3t^3+r_4t^2+r_5t+r_6$. If $f(t)$ has only positive and real roots, then the following holds
    \begin{enumerate}
        \item[(1)] $r_1<0$;
        \item[(2)] $r_2>0$ and $r_2\leq \frac{5}{12}r_1^2$;
        \item[(3)] $r_3<0$ and $-\frac{1}{3}r_1^2\sqrt{25r_1^2 - 60r_2}+\frac{1}{225}(25r_1^2 - 60r_2)^{\frac{3}{2}}+\frac{4}{5}r_2\sqrt{25r_1^2 - 60r_2}\leq \frac{10}{9}r_1^3 -4r_1r_2+6r_3\leq \frac{1}{3}r_1^2\sqrt{25r_1^2 - 60r_2}-\frac{1}{225}(25r_1^2 - 60r_2)^{\frac{3}{2}}-\frac{4}{5}r_2\sqrt{25r_1^2 - 60r_2}$;
        \item[(4)] $r_4>0$ and $L_1\leq r_4\leq R_1$;
        \item[(5)] $r_5<0$ and $L_2\leq r_5\leq R_2$;
    \end{enumerate}
    where $L_1,L_2, R_1,R_2$ are defined in the proof.
\end{lemma}
\begin{proof}
By Rolle's theorem, if $f(t)$ has only positive and real roots, then so do all of its derivatives. 
In particular, if we let
 $$g(t) :=\frac{1}{6}f'(t) = t^5+ \frac{5r_1}{6}t^4+\frac{2r_2}{3}t^3+ \frac{r_3}{2}t^2+\frac{r_4}{3}t+\frac{r_5}{6},$$ then $g(t)$ is a monic polynomial with only positive and real roots.
We denote by 
\begin{equation}\label{eq:giri}
    g_1 := \frac{5r_1}{6},\: g_2 := \frac{2r_2}{3},\: g_3 := \frac{r_3}{2},\: g_4 := \frac{r_4}{3},\: g_5 :=\frac{r_5}{6}.
\end{equation}

We will now apply \cite[Corollary 3.6]{marseglia2025} on $g(t)$ to get bounds on its coefficients $g_i$ for $i=1,\dots,5$ which in turn will give bounds on $r_i$ for $i=1,\dots,5$.
Firstly, we need to define the following quantities \begin{align*}
    u_2 &:= \frac{3g_2}{10} - \frac{3g_1^2}{25}=- \frac{1}{12} r_1^2 + \frac{1}{5} r_2, \\
    u_3 &:= \frac{2g_1^3}{125} - \frac{3g_1 g_2}{50} + \frac{g_3}{10}=\frac{1}{108} r_1^3 - \frac{1}{30} r_1 r_2 + \frac{1}{20} r_3, \\
    u_4 &:= -\frac{3g_1^4}{625} + \frac{3g_1^2 g_2}{125} - \frac{2g_1 g_3}{25} + \frac{g_4}{5}=- \frac{1}{432} r_1^4 + \frac{1}{90} r_1^2 r_2 - \frac{1}{30} r_1 r_3 + \frac{1}{15} r_4, \\
\end{align*}

Define
    \[ S \coloneqq \left\{-\zeta^{k} \omega \left( \frac{9 u_3}{2} + \frac{3\sqrt{\Delta}}{2}\right) - \zeta^{-k} \overline{\omega} \left( \frac{9 u_3}{2} - \frac{3\sqrt{\Delta}}{2} \right)  + \frac{2 u_2^2}{3} : 0\leq k \leq 2 \right\}, \]
    where $\Delta \coloneqq u_3^2 + \frac{4}{27} u_2^3$, $\zeta$ is a primitive third root of unity and $\omega$ is a third root of $ \frac{1}{2} (- u_3 + \sqrt{\Delta})$.
By \cite[Corollary 3.6]{marseglia2025}, if all the roots of $g(x)$ are real or (1)--(4) hold then the set $S$ consists of real numbers, which we denote by $\theta_1 \leq \theta_2 \leq \theta_3$.
If \( u_3 = 0 \), let \( x_{i_1, i_2} \) be
$$x_{i_1, i_2} := i_1 \sqrt{-u_2 + i_2 \sqrt{u_2^2 - u_4}},$$
for \( i_1, i_2 \in \{+1,-1\} \). Otherwise, define
$$v_2 := -\frac{u_2^2}{3} - u_4,$$
$$v_3 := \frac{2 u_2 u_4}{3} - \frac{2 u_2^3}{27} - 2 u_3^2,$$
$$\text{$C$ is a third root of }
        \begin{cases}
            -v_3 & \text{if $v_2=0$},\\
            \frac{1}{2}\left(- v_3 + \sqrt{v_3^2 + \frac{4}{27} v_2^3}\right) & \text{if $v_2\neq 0$},
        \end{cases}$$
$$y \coloneqq C - \frac{v_2}{3C} - \frac{2 u_2}{3}
$$
and let \( x_{i_1, i_2} \) for \( i_1, i_2 \in \{+1,-1\} \) be defined by
$$x_{i_1, i_2} := \frac{i_1 \sqrt{2y} + i_2 \sqrt{-4u_2-2y-i_1\frac{8u_3}{\sqrt{2y}}}}{2}.$$

Then, by \cite[Proposition 3.5]{marseglia2025}, $x_{i_1, i_2} - \frac{g_4}{5}$, for $i_1, i_2 \in \{+1, -1 \}$ are the four roots of $g'(x)$.
    Assume that they are real, sort them as $\beta_4 \leq \beta_3 \leq \beta_2 \leq \beta_1$, set
    \begin{align*}
        \lambda_1 & \coloneqq \max\left\lbrace \beta_1^4+g_2 \beta_1^3+g_3\beta_1^2+g_4\beta_1, 
        \beta_3^4+g_2\beta_3^3+g_3\beta_3^2+g_4\beta_3 \right\rbrace \\
        \lambda_2 & \coloneqq \min\left\lbrace \beta_2^4+g_2\beta_2^3+g_3\beta_2^2+g_4\beta_2, 
        \beta_4^4+g_2\beta_4^3+g_3\beta_4^2+g_4\beta_4 \right\rbrace.
    \end{align*}

By using the conclusion of \cite[Proposition 3.5, Corollary 3.6]{marseglia2025} applied to $g(t)$ and making the substitutions given in \eqref{eq:giri} we get the conclusion. Here, we define $L_1,R_1,L_2,R_2$ as follows:
$$L_1:=
\frac{5}{144} r_1^4 - \frac{1}{6} r_1^2 r_2 + \frac{1}{2} r_1 r_3
+15\theta_1,
$$
$$
R_1:=\frac{5}{144} r_1^4 - \frac{1}{6} r_1^2 r_2 + \frac{1}{2} r_1 r_3
+15\theta_2,
$$
$$L_2: = -6\lambda_2, \qquad R_2:=-6\lambda_1.$$

\end{proof}
\begin{remark}\label{rmk:bounds}
    Let $\chi(t) = t^{12}+a_1t^{11}+\cdots+a_6t^6+qa_5t^5+\cdots+q^5a_1t+q^6$. By Vieta's relations and the fact that the roots are $q$-Weil numbers, it follows that each $a_i$ satisfies the trivial bound
    \[
    |a_i|\leq \binom{12}{i} (\sqrt{q})^i.
    \]
    Moreover, the equality holds only when $\chi(t)$ all roots in $\R$.
\end{remark}
\begin{corollary}
    Let $\chi(t) = t^{12}+a_1t^{11}+\cdots+a_6t^6+qa_5t^5+\cdots+q^5a_1t+q^6$. Assume it has no real roots. If $\chi(t)$ is a Weil polynomial, then the following holds

    \begin{enumerate}
        \item $-12\sqrt{q}< a_1 <12\sqrt{q}$;
        \item $-54q+10\sqrt{q}|a_1|<a_2\leq 6q +\frac{5}{12}a_1^2$; 
        \item $-35qa_1+8\sqrt{q}a_2-112q\sqrt{q}<a_3<-35qa_1+8\sqrt{q}a_2+112q\sqrt{q}$;
        \item $a_3\geq (-1/27a_1^2 + 4/45a_2 - 8/15q)\sqrt{25a_1^2 - 60a_2 + 360q} - 5/27a_1^3 + 2/3a_1a_2 + q a_1$,\\
        $a_3\leq (1/27a_1^2 - 4/45a_2 + 8/15q)\sqrt{25a_1^2 - 60a_2 + 360q}- 5/27a_1^3 + 2/3a_1a_2 + q a_1 $; 
        \item $2\sqrt{q}|25qa_1+3a_3|-105q^2-20qa_2<a_4$ ;
        \item $L_4\leq a_4\leq R_4$;
        \item $-36q^2\sqrt{q}-25q^2a_1-16q\sqrt{q}a_2-9qa_3-4\sqrt{q}a_4<a_5<36q^2\sqrt{q}-25q^2a_1+16q\sqrt{q}a_2-9qa_3+4\sqrt{q}a_4$; 
        \item $L_5\leq a_5\leq R_5$; 
        \item $-924q^3< a_6 < 924 q^3$. 
    \end{enumerate}
The description of the bounds $L_4,L_5,R_4,R_5$ in 6. and 8. is as follows.
We first define the following quantities: \begin{align*}
    u_2 &:=  -\frac{a_1^2}{12}+\frac{a_2}{5}-\frac{6q}{5},\\
    u_3 &:=\frac{a_1^3}{108}  - \frac{a_1 a_2}{30}  - \frac{a_1q}{20}   + \frac{a_3}{20},  \\
    u_4 &:=  -\frac{a_1^4}{432}+\frac{a_1^2a_2}{90}+\frac{qa_1^2}{10}-\frac{a_1a_3}{30}-\frac{4qa_2}{15}+\frac{a_4}{15}+\frac{3q^2}{5},\\
\end{align*}
Moreover, we define $\theta_1,\: \theta_2, \:\lambda_1,\:\lambda_2$ as in the proof of Lemma~\ref{lem:ineq}, by plugging in the values of $u_1,\:u_2,\:u_3$ above.
Finally, we define
\begin{align*}
L_4 &:= \frac{5}{144}a_1^4 - \frac{1}{6}a_1^2 a_2 - \frac{3}{2}a_1^2 q + \frac{1}{2}a_1 a_3 + 4 a_2 q - 9 q^2 + 15\theta_1, \\
R_4 &:= \frac{5}{144}a_1^4 - \frac{1}{6}a_1^2 a_2 - \frac{3}{2}a_1^2 q + \frac{1}{2}a_1 a_3 + 4 a_2 q - 9 q^2 + 15\theta_2, \\
L_5 &:= -36q^2\sqrt{q} - 25q^2a_1 - 16q\sqrt{q}a_2 - 9q a_3 - 4\sqrt{q}a_4 + 6\lambda_1, \\
R_5 &:= \phantom{-}36q^2\sqrt{q} - 25q^2a_1 + 16q\sqrt{q}a_2 - 9q a_3 + 4\sqrt{q}a_4 - 6\lambda_2.
\end{align*}

\end{corollary}
\begin{proof}
The bounds on $a_i$ with $0\leq i\leq5$ come from Lemma~\ref{lem:ineq} applied to the polynomials \eqref{eq:f} and \eqref{eq:tildef}.
For $a_6$, we use the trivial bound given in Remark \ref{rmk:bounds}.
   
\end{proof}
\section{Question \ref{Q2} for $g=7$}\label{sec:Q2}
In this section we let $A$ be a simple abelian variety of dimension $g$ over $\F_q$ with $q = p^n$, for a prime $p$ and integer $n\geq 1$. We denote by $\Frob_A$
the Frobenius endomorphism of $A$ and $\chi_A(t)$ the characteristic polynomial
of $\Frob_A$ -- as viewed inside $V_l(A)$, for an $l\neq p$ as described in Section \ref{intro}. We recall that
\begin{equation}
    \chi_A(t)=m_A(t)^e,
\end{equation}
where $m_A(t)$ is an irreducible polynomial, and $e\geq 1$ an integer, called the multiplicity of $A$. From this equality, it is obvious that $e\mid2\dim(A)$. As discussed in the introduction, it is always the case that $\chi_A(t)$ is a Weil polynomial. However, the converse is not always true, and it is covered by Question \ref{Q2}.
In this section, we answer this question for $g=7$ in the following theorem. 
For an integer $a\in \Z$ denote by $v_p(a)$ the $p$-adic valuation of $a$.
\begin{theorem}\label{Q2g7}
    Let $f(t) = t^{14}+ a_1t^{13}+a_2t^{12}+ \ldots +a_2q^5t^2+a_1q^6t+q^7$
be an irreducible Weil polynomial. Then the polynomial $f(t)$ is the
characteristic polynomial of a simple abelian variety of dimension $7$ over $\F_q$ if and only
if one of the following conditions holds
\begin{enumerate}
    \item $v_p(a_1)\geq \frac{1}{2}n, 
        v_p(a_2)\geq n, 
        v_p(a_3)\geq \frac{3}{2}n, 
        v_p(a_4)\geq 2n, 
        v_p(a_5)\geq \frac{5}{2}n, 
        v_p(a_6)\geq 3n, 
        v_p(a_7)\geq \frac{7}{2}n$ and $f(t)\text{ has no root of valuation }\frac{1}{2}n\text{ nor irreducible factors of degree }3\text{ or }5\text{ or }7\text{ in }\mathbb{Q}_p[t];$
    \item $v_p(a_1)=0, v_p(a_2)\geq \frac{1}{2}n,
        v_p(a_3)\geq n , 
        v_p(a_4)\geq \frac{3}{2}n , 
        v_p(a_5)\geq 2n , 
        v_p(a_6)\geq \frac{5}{2}n , 
        v_p(a_7)\geq 3n$ and $f(t)\text{ has no root of valuation }\frac{1}{2}n\text{ nor irreducible factors of degree }3\text{ or }5\text{ in }\mathbb{Q}_p[t];$
    \item $v_p(a_1)=0, v_p(a_2)\geq \frac{1}{3}n,
        v_p(a_3)\geq \frac{2}{3}n,
        v_p(a_5)\geq \frac{3}{2}n , 
        v_p(a_4)=n,
        v_p(a_6)\geq 2n , 
        v_p(a_7)\geq \frac{5}{2}n$
        and $f(t)\text{ has no root of valuation }\frac{1}{3}n,\frac{1}{2}n\text{ or }\frac{2}{3}n\text{ and has exactly }2\text{ degree }3\text{ irreducible factors in }\mathbb{Q}_p[t];$
    \item  $v_p(a_1)=0, v_p(a_2)\geq \frac{1}{4}n , 
        v_p(a_3)\geq \frac{1}{2}n , 
        v_p(a_4)\geq \frac{3}{4}n , 
        v_p(a_5)=n,
        v_p(a_6)\geq \frac{3}{2}n , 
        v_p(a_7)\geq 2n$ and $f(t)\text{ has no root of valuation }\frac{1}{4}n,\frac{1}{2}n\text{ or }\frac{3}{4}n\text{ and has at most }2\text{ degree }2\text{ irreducible factors in }\mathbb{Q}_p[t];$
    \item $v_p(a_1)=0, v_p(a_2)\geq \frac{2}{5}n , 
        v_p(a_3)\geq \frac{4}{5}n , 
        v_p(a_4)\geq \frac{6}{5}n , 
        v_p(a_5)\geq \frac{8}{5}n , 
        v_p(a_6)= 2n,
        v_p(a_7)\geq \frac{5}{2}n$ and $f(t)\text{ has no root of valuation }\frac{2}{5}n,\frac{1}{2}n\text{ or }\frac{3}{5}n\text{ nor degree }3\text{ irreducible factors in }\mathbb{Q}_p[t];$
    \item $v_p(a_1)=0, v_p(a_2)\geq \frac{1}{5}n , 
        v_p(a_3)\geq \frac{2}{5}n , 
        v_p(a_4)\geq \frac{3}{5}n , 
        v_p(a_5)\geq \frac{4}{5}n , 
        v_p(a_6)=n,
        v_p(a_7)\geq \frac{3}{2}n$ and $f(t)\text{ has no root of valuation }\frac{1}{5}n,\frac{1}{2}n\text{ or }\frac{4}{5}n\text{ nor degree }3\text{ irreducible factors in }\mathbb{Q}_p[t];$
    \item $v_p(a_1)=0, v_p(a_2)\geq \frac{1}{3}n , 
        v_p(a_3)\geq \frac{2}{3}n , 
        v_p(a_4)\geq n , 
        v_p(a_5)\geq \frac{4}{3}n , 
        v_p(a_6)\geq \frac{5}{3}n,
        v_p(a_7)=2n$ and $f(t)\text{ has no root of valuation }\frac{1}{3}n\text{ or }\frac{2}{3}n\text{ nor degree }2\text{ irreducible factors in }\mathbb{Q}_p[t];$
    \item $v_p(a_1)=0, v_p(a_2)\geq \frac{1}{6}n , 
        v_p(a_3)\geq \frac{1}{3}n , 
        v_p(a_4)\geq \frac{1}{2}n , 
        v_p(a_5)\geq \frac{2}{3}n , 
        v_p(a_6)\geq \frac{5}{6}n,
        v_p(a_7)=n $ and $f(t)\text{ has no root of valuation $\frac{1}{6}n$ or $\frac{5}{6}n$ nor irreducible factors of degree $2$ or $3$ in }\mathbb{Q}_p[t];$
    \item  $v_p(a_1)\geq 0,
    v_p(a_2)=0,
    v_p(a_3)\geq n/2,
    v_p(a_4)\geq n,
    v_p(a_5)\geq 3n/2,
    vp(a_6)\geq 2n,
    v_p(a_7)\geq 5n/2$ and $f(t)\text{ has no root of valuation } \frac{1}{2}n \text{ nor irreducible factors of degree $3$ or $5$ in }\mathbb{Q}_p[t];$
    \item  $v_p(a_1)\geq 0, 
    v_p(a_2)=0, 
    v_p(a_3)\geq n/3, 
    v_p(a_4)\geq 2n/3, 
    v_p(a_5)\geq n, 
    v_p(a_6)\geq 3n/2, 
    v_p(a_7)\geq 2n$ and $f(t)\text{ has no root of valuation } \frac{1}{2}n, \frac{1}{3}n, \frac{2}{3}n \text{ over }\Q_p[t] $;
    \item  $v_p(a_1)\geq 0, 
    v_p(a_2)=0, 
    v_p(a_3)\geq n/4, 
    v_p(a_4)\geq n/2, 
    v_p(a_5)\geq 3n/4, 
    v_p(a_6)\geq n, 
    v_p(a_7)\geq 5n/4$ and $f(t)\text{ has no root of valuation } \frac{1}{2}n, \frac{3}{4}n, \frac{1}{4}n \text{ and we cannot have more than 3 irreducible factors of degree 2}$;
    \item  $v_p(a_1)\geq 0, 
    v_p(a_2)=0,
    v_p(a_3)\geq \frac{1}{5}n, 
    v_p(a_4)\geq \frac{2}{5}n, 
    v_p(a_5)\geq \frac{3}{5}n, 
    v_p(a_6)\geq \frac{4}{5}n, 
    v_p(a_7)=n$ and $f(t) \text{ has no root of valuation $\frac{1}{5}n$ or $\frac{4}{5}n$ and has no irreducible factors of degree 3 in $\QQ_p[t]$}$;
    \item $ v_p(a_1)\geq 0, 
    v_p(a_2)=0,
    v_p(a_3)\geq \frac{2}{5}n, 
    v_p(a_4)\geq \frac{4}{5}n, 
    v_p(a_5)\geq \frac{6}{5}n, 
    v_p(a_6)\geq \frac{8}{5}n, 
    v_p(a_7)=2n$ and $f(t) \text{ has no root of valuation $\frac{2}{5}n$ or $\frac{3}{5}n$ and has no irreducible factors of degree 3 in $\QQ_p[t]$}$;
    \item  $v_p(a_1)\geq 0 , 
        v_p(a_2)\geq 0 , 
        v_p(a_3) = 0 , 
        v_p(a_4)\geq \frac{1}{2}n , 
        v_p(a_5)\geq n , 
        v_p(a_6)\geq \frac{3}{2}n , 
        v_p(a_7)\geq 2n$ and $f(t)\text{ has no root of valuation $\frac{1}{2}n$ nor irreducible factors of degree $3$ in $\mathbb{Q}_p[t]$}$;
    \item  $v_p(a_1)\geq 0 , 
        v_p(a_2)\geq 0 , 
        v_p(a_3) = 0 , 
        v_p(a_4)\geq \frac{1}{3}n , 
        v_p(a_5)\geq \frac{2}{3}n , 
        v_p(a_6) = n, 
        v_p(a_7)\geq \frac{3}{2}n$ and $f(t)\text{ has no root of valuation $\frac{1}{3}n$, $\frac{1}{2}n$ or $\frac{2}{3}n$ in $\mathbb{Q}_p[t]$}$;
    \item  $v_p(a_1)\geq 0 , 
        v_p(a_2)\geq 0 , 
        v_p(a_3) = 0 , 
        v_p(a_4)\geq \frac{1}{4}n , 
        v_p(a_5)\geq \frac{1}{2}n , 
        v_p(a_6)\geq \frac{3}{4}n , 
        v_p(a_7) = n$ and $f(t) \text{ has no root }\text{of valuation $\frac{1}{4}n$ or $\frac{3}{4}n$ and has at most $2$ irreducible factors of degree $2$ in $\mathbb{Q}_p[t]$}.$
    \item $v_p(a_1)\geq \frac{1}{3}n , 
        v_p(a_2)\geq \frac{2}{3}n , 
        v_p(a_3) = 0 , 
        v_p(a_4)\geq \frac{3}{2}n , 
        v_p(a_5)\geq 2n , 
        v_p(a_6)\geq \frac{5}{2}n , 
        v_p(a_7)\geq 3n$ and $f(t)\text{ has no root of valuation $\frac{1}{3}n$, $\frac{1}{2}n$ or $\frac{2}{3}n$ and has at most $2$ irreducible factors of degree $3$ in $\mathbb{Q}_p[t]$}$;
    \item  $v_p(a_1)\geq \frac{1}{4}n,
    v_p(a_2)\geq \frac{1}{2}n,
    v_p(a_3)\geq \frac{3}{4}n,
    v_p(a_4)=n,
    v_p(a_5)\geq \frac{3}{2}n,
    v_p(a_6)\geq 2n,
    v_p(a_7)\geq \frac{5}{2}n$ and $f(t)$ has no root of valuation $\frac{1}{4}n, \frac{1}{2}n$ or $\frac{3}{4}n$ and has no irreducible factors of degree $3$
    and $f(t)$ has at most $3$ irreducible factors of degree 2 in $\mathbb{Q}_p[t]$;
    \item $v_p(a_1)\geq \frac{1}{4}n, , v_p(a_2)\geq \frac{1}{2}n,, v_p(a_3)\geq \frac{3}{4}n, , v_p(a_4) = n
    v_p(a_5)\geq \frac{4}{3}n,
     v_p(a_6)\geq \frac{5}{3}n,
     v_p(a_7) = 2n$ and $f(t)$ has no root of valuation $\frac{3}{4}n, \frac{2}{3}n,\frac{1}{3}n,\frac{1}{4}n$ in $\QQ_p$, nor a factor of degree 2 in $\QQ_p[t]$.
    \item $v_p(a_1)\geq 0,, v_p(a_2)\geq 0,, v_p(a_3)\geq 0, v_p(a_4)=0, v_p(a_5)\geq \frac{1}{2}n,, v_p(a_6)\geq n,, v_p(a_7)\geq \frac{3}{2}n$ and $f(t)$ has no root of valuation $\frac{1}{2}n$ in $\QQ_p$, nor a factor of degree 3 in $\QQ_p[t]$;
    \item $v_p(a_1)\geq 0, \quad v_p(a_2)\geq 0,\quad v_p(a_3)\geq 0,\quad v_p(a_4) = 0, v_p(a_5)\geq \frac{1}{3}n,
    v_p(a_6)\geq \frac{2}{5}n, v_p(a_7) = n$ and $f(t)$ has no root of valuation $\frac{1}{3}n, \frac{2}{3}n$ in $\QQ_p$;
    \item 
$v_p(a_1)\geq 0  ,
        v_p(a_2)\geq 0  ,
        v_p(a_3)\geq 0  ,
        v_p(a_4)\geq 0  ,
        v_p(a_5)\geq 0  ,
        v_p(a_6)\geq \frac{1}{2}n  ,
        v_p(a_7)\geq n$ and $f(t)\text{ has no root of valuation $\frac{1}{2}n$ in $\mathbb{Q}_p[t]$}$;
    \item $v_p(a_1)\geq \frac{1}{5}n  ,
        v_p(a_2)\geq \frac{2}{5}n  ,
        v_p(a_3)\geq \frac{3}{5}n  ,
        v_p(a_4)\geq \frac{4}{5}n  ,
        v_p(a_5)\geq 0  ,
        v_p(a_6)\geq \frac{3}{2}n  ,
        v_p(a_7)\geq 2n$ and $f(t)\text{ has no root of valuation $\frac{1}{5}n$, $\frac{1}{2}n$ or $\frac{4}{5}n$ nor degree $3$ irreducible factors in $\mathbb{Q}_p[t]$}$;
    \item $ v_p(a_1)\geq \frac{2}{5}n  ,
        v_p(a_2)\geq \frac{4}{5}n  ,
        v_p(a_3)\geq \frac{6}{5}n  ,
        v_p(a_4)\geq \frac{8}{5}n  ,
        v_p(a_5)\geq 0  ,
        v_p(a_6)\geq \frac{5}{2}n  ,
        v_p(a_7)\geq 3n$ and $f(t)\text{ has no root of valuation $\frac{2}{5}n$, $\frac{1}{2}n$ or $\frac{3}{5}n$ nor degree $3$ irreducible factors in $\mathbb{Q}_p[t]$}$;
    \item $v_p(a_1)\geq 0  ,
        v_p(a_2)\geq 0  ,
        v_p(a_3)\geq 0  ,
        v_p(a_4)\geq 0  ,
        v_p(a_5)\geq 0  ,
        v_p(a_6)\geq 0  ,
        v_p(a_7)\geq \frac{1}{2}n$ and $f(t)\text{ has no root of valuation }\frac{1}{2}n\text{ in }\mathbb{Q}_p[t].$
    \item $v_p(a_1)\geq \frac{1}{6}n  ,
        v_p(a_2)\geq \frac{1}{3}n  ,
        v_p(a_3)\geq \frac{1}{2}n  ,
        v_p(a_4)\geq \frac{2}{3}n  ,
        v_p(a_5)\geq \frac{5}{6}n  ,
        v_p(a_6)\geq 0  ,
        v_p(a_7)\geq \frac{3}{2}n$ and $f(t)$ has no root of valuation $\frac{1}{6}n, \frac{1}{2}n$ or $\frac{5}{6}n$ nor irreducible factors of degree $3$
       and $f(t)$ has exactly $1$ irreducible factor of degree 2 in $\mathbb{Q}_p[t]$;
    \item  $v_p(a_1)\geq \frac{1}{3}n  ,
        v_p(a_2)\geq \frac{2}{3}n  ,
        v_p(a_3)\geq n  ,
        v_p(a_4)\geq \frac{4}{3}n  ,
        v_p(a_5)\geq \frac{5}{3}n  ,
        v_p(a_6)\geq 0  ,
        v_p(a_7)\geq \frac{5}{2}n$ and $ f(t) \text{ has no root }\text{of valuation $\frac{1}{3}n, \frac{1}{2}n$ or $\frac{2}{3}n$ and has exactly $1$ irreducible factor of degree 2 in $\mathbb{Q}_p[t]$}$;
    \item  $v_p(a_1)\geq 0  ,
        v_p(a_2)\geq 0  ,
        v_p(a_3)\geq 0  ,
        v_p(a_4)\geq 0  ,
        v_p(a_5)\geq 0  ,
        v_p(a_6)\geq 0  ,
        v_p(a_7)\geq \frac{1}{2}n$;
    \item $v_p(a_1)\geq \frac{1}{7}n ,
        v_p(a_2)\geq \frac{2}{7}n  ,
        v_p(a_3)\geq \frac{3}{7}n  ,
        v_p(a_4)\geq \frac{4}{7}n  ,
        v_p(a_5)\geq \frac{5}{7}n  ,
        v_p(a_6)\geq \frac{6}{7}n  ,
        v_p(a_7)\geq \frac{3}{2}n$ and $f(t) \text{ has no root of valuation $\frac{1}{7}n$ or $\frac{6}{7}n$ nor irreducible factors of degree $2$ and $3$}$;
    \item $v_p(a_1)\geq \frac{2}{7}n  ,
        v_p(a_2)\geq \frac{4}{7}n  ,
        v_p(a_3)\geq \frac{6}{7}n  ,
        v_p(a_4)\geq \frac{8}{7}n  ,
        v_p(a_5)\geq \frac{10}{7}n  ,
        v_p(a_6)\geq \frac{12}{7}n  ,
        v_p(a_7)\geq \frac{5}{2}n$ and $ f(t) \text{ has no root of valuation $\frac{2}{7}n$ or $\frac{5}{7}n$ nor irreducible factors of degree $2$ and $3$};$
    \item $ v_p(a_1)\geq \frac{3}{7}n  ,
        v_p(a_2)\geq \frac{6}{7}n  ,
        v_p(a_3)\geq \frac{9}{7}n  ,
        v_p(a_4)\geq \frac{12}{7}n  ,
        v_p(a_5)\geq \frac{15}{7}n  ,
        v_p(a_6)\geq \frac{18}{7}n  ,
        v_p(a_7)\geq \frac{5}{2}n$ and $ f(t) \text{ has no root of valuation $\frac{3}{7}n$ or $\frac{4}{7}n$ nor irreducible factors of degree $2$ and $3$}.$
\end{enumerate}
\end{theorem}
Next we proceed to prove the above theorem.
\subsection{Theoretical background}
We collect a few theoretical results from Hayashisa's \cite{hay1} which we will use to prove Theorem \ref{Q2g7}.
\begin{lemma}\label{noreal}
Let $A$ be a simple abelian variety over $\F_q$ with $q = p^n$ elements and $\chi_A(t)$
the characteristic polynomial of $\Frob_A$. Suppose that $\chi_A(t)$ has a real root. Then we have:
\begin{itemize}
    \item if $n$ is even, then $\dim(A) = 1$,
    \item if $n$ is odd, then $\dim(A) = 2$.
\end{itemize}
\end{lemma}
\begin{proof}
    This is Lemma 2.4. in \cite{hay1}.
\end{proof}
\begin{corollary}\label{cor:prime}
    Recall that the \textit{multiplicity} of $A$ is the integer $e\geq 1$ satisfying $\chi_A(t)=m_A(t)^e$. Suppose that the dimension of $A$ is a prime number $\ell\geq 3$. Then $e=1$ or $e=l$.
\end{corollary}

\begin{proof}
    Since $e$ divides $2\ell$ and $\ell$ is a prime, we have either $e \in \{ 1, 2, \ell, 2\ell\}$. Suppose $e = 2$
or $e = 2\ell$. Then $m_A(t)$ is an irreducible polynomial of odd degree. Hence the polynomial
$m_A(t)$ has a real root. This contradicts $\ell\geq 3$ by Lemma \ref{noreal}.
\end{proof}
We note that $e=\ell$ is covered in the following theorem. So we are left with studying $e=1$, i.e. the irreducible case.
\begin{theorem}
    Let $a, b \in \mathbb{Z}$ and $2 < g \in \mathbb{Z}$. Set $f(t) = (t^2 + at + b)^g \in \mathbb{Z}[t]$. Then the
polynomial $f(t)$ is the characteristic polynomial of a simple abelian variety of dimension
$g$ over $\F_q$ with $q = p^n$ elements if and only if $g$ divides $n$, $b = q$, $|a| < 2\sqrt{q} $ and $a = kq^{s/g}$,
where $k, s$ are integers satisfying $\gcd(k, p) = 1, \gcd(s, g) = 1$ and $1 \leq s < g/2$.
\end{theorem}
\begin{proof}
    This is Theorem 1.2. in \cite{hay1}.
\end{proof}
\begin{theorem}\label{REF}
    An irreducible Weil polynomial $f(t)$ of degree $2g$ is the characteristic polynomial of a simple abelian variety of dimension $g$ over $\F_q$ (i.e. $e = 1$) if and only if $f(t)$
has no real roots and the following condition holds
\begin{equation}
    \frac{v_p(f_i(0))}{n}\in \Z, \text{ for all }f_i\text{ monic irreducible factors of }f(t)\text{ base changed to } \Q_p.
\end{equation}

\end{theorem}
\begin{proof}
    This is Corollary 3.2. in \cite{hay1} which follows from the work of Tate (e.g. Theorem 8 in \cite{WaterhouseMilne1969}).
\end{proof}

\subsection{Proof of Theorem \ref{Q2g7}}

\begin{proof}

Following Hayashida's proof of \cite[Theorem 1.3.]{hay1} we let $f(t)$ be an irreducible Weil polynomial of
degree $14$. It is the characteristic polynomial of a simple abelian variety $A$ of dimension $7$
over $\F_q$ if and only if the conditions in Theorem \ref{REF} hold. Firstly, note that if $f(t)$ comes from a simple abelian variety $A$ and has a real root, then Lemma \ref{noreal} gives that $\dim(A)$ is $1$ or $2$, a contradiction. To check the second condition, we employ Newton Polygons.
Let $\mathcal{NP}(f)$ be the Newton Polygon of $f$. Then $\mathcal{NP}(f)$ has $14$ possible vertices given by the following coordinates:
\[
(0,7n), (1,6n+v_p(a_1)), (2, 5n+v_p(a_2)), (3, 4n+v_p(a_3)), (4, 3n+v_p(a_4)), (5, 2n+v_p(a_5)), (6, n+v_p(a_6)), 
\]
\[
(7, v_p(a_7)),(8, v_p(a_6)), (9,v_p(a_5)), (10, v_p(a_4)), (11, v_p(a_3)), (12, v_p(a_2)), (13, v_p(a_1)), (14, 0).
\]
They give rise to $31$ possible Newton Polygons, sandwiched between the two boundary cases of an ordinary simple abelian variety and a supersingular one.
Hayashida notes that if some of these points
are vertices, then the point must be a lattice point belonging to $\Z \times \Z/n\Z$ (\cite{hay1}). By symmetry of $\mathcal{NP}(f)$, it is sufficient to classify the possible Newton polygons according to whether either of
\[
(0,7n), (1,6n+v_p(a_1)), (2, 5n+v_p(a_2)), (3, 4n+v_p(a_3)), (4, 3n+v_p(a_4)),\]
\[(5, 2n+v_p(a_5)), (6, n+v_p(a_6)), (7, v_p(a_7)),
\]
is a vertex. The remainder of the proof involves detailing each of these $31$ possible polygons, which are depicted in red in the Figure~\ref{fig.all}. For clarity, we present the case where there is no vertex (i.e. the supersingular case), along with the one where the initial vertex is given by $(1,6n+v_p(a_1))$, noting that the other cases follow similarly.\\
\textbf{Case $0$.} Assume there is no vertex. Then
    $$\begin{cases}
        6n+v_p(a_1)\geq \frac{13}{2}n,\\
        5n+v_p(a_2)\geq 6n, \\
        4n+v_p(a_3)\geq \frac{11}{2}n, \\
        3n+v_p(a_4)\geq 5n, \\
        2n+v_p(a_5)\geq \frac{9}{2}n, \\
        n+v_p(a_6)\geq 4n, \\
        v_p(a_7)\geq \frac{7}{2}n,
    \end{cases}\Rightarrow \begin{cases}
        v_p(a_1)\geq \frac{1}{2}n, \\
        v_p(a_2)\geq n,\\
        v_p(a_3)\geq \frac{3}{2}n, \\
        v_p(a_4)\geq 2n, \\
        v_p(a_5)\geq \frac{5}{2}n, \\
        v_p(a_6)\geq 3n, \\
        v_p(a_7)\geq \frac{7}{2}n.
    \end{cases}$$
    In this case, we have
    $$(t-\alpha_1)\cdots(t-\alpha_{14})\in\mathbb{Q}_p[t]$$
    with
    \begin{align*}
        v_p(\alpha_1)&= \cdots = v_p(\alpha_{14})=\frac{1}{2}n.\\
    \end{align*}
    Thus Theorem \ref{REF} holds if and only if
    $f(t)$ has no root of valuation $\frac{1}{2}n$ nor irreducible factors of degree $3$, $5$, or $7$ in $\mathbb{Q}_p[t]$.
    
\textbf{Case 1.} Assume the first vertex is $(1,6n+v_p(a_1))$. In this case, the only possibility is $(1,6n+v_p(a_1))=(1,6n)$ and $v_p(a_1)=0$. Then we land in the following subcases.

\begin{enumerate}
    \item [\textbf{ 1.1}] Suppose that $(1,6n+v_p(a_1))=(1,6n)$ is the sole vertex. Then
    $$\begin{cases}
        5n+v_p(a_2)\geq \frac{11}{2}n, \\
        4n+v_p(a_3)\geq 5n, \\
        3n+v_p(a_4)\geq \frac{9}{2}n, \\
        2n+v_p(a_5)\geq 4n, \\
        n+v_p(a_6)\geq \frac{7}{2}n, \\
        v_p(a_7)\geq 3n,
    \end{cases}\Rightarrow \begin{cases}
        v_p(a_2)\geq \frac{1}{2}n, \\
        v_p(a_3)\geq n ,\\
        v_p(a_4)\geq \frac{3}{2}n, \\
        v_p(a_5)\geq 2n ,\\
        v_p(a_6)\geq \frac{5}{2}n, \\
        v_p(a_7)\geq 3n.
    \end{cases}$$
    In this case, we have
    $$t-\alpha_1,(t-\alpha_2)\cdots(t-\alpha_{13}),t-\alpha_{14}\in\mathbb{Q}_p[t]$$
    with
    \begin{align*}
        v_p(\alpha_1)&=n,\\
        v_p(\alpha_2)&= \cdots = v_p(\alpha_{13})=\frac{1}{2}n,\\
        v_p(\alpha_{14})&= 0.
    \end{align*}
    Thus Theorem \ref{REF} holds if and only if
    $$f(t)\text{ has no root of valuation }\frac{1}{2}n\text{ nor irreducible factors of degree }3\text{ or }5\text{ in }\mathbb{Q}_p[t].$$
    
    \item [\textbf{ 1.2}] Suppose that $(1,6n+v_p(a_1))=(1,6n)$ and $(4,3n+v_p(a_4))=(4,4n)$ are vertices. Then $v_p(a_4)=n$ and
    $$\begin{cases}
        5n+v_p(a_2)\geq \frac{16}{3}n, \\
        4n+v_p(a_3)\geq \frac{14}{3}n, \\
        2n+v_p(a_5)\geq \frac{7}{2}n, \\
        n+v_p(a_6)\geq 3n,\\
        v_p(a_7)\geq \frac{5}{2}n,
    \end{cases}\Rightarrow \begin{cases}
        v_p(a_2)\geq \frac{1}{3}n, \\
        v_p(a_3)\geq \frac{2}{3}n, \\
        v_p(a_5)\geq \frac{3}{2}n, \\
        v_p(a_6)\geq 2n ,\\
        v_p(a_7)\geq \frac{5}{2}n.
    \end{cases}$$
    In this case, we have
    $$t-\alpha_1,(t-\alpha_2)\cdots(t-\alpha_4),(t-\alpha_5)\cdots(t-\alpha_{10}),(t-\alpha_{11})\cdots(t-\alpha_{13}),t-\alpha_{14}\in\mathbb{Q}_p[t]$$
    with
    \begin{align*}
        v_p(\alpha_1)&=n,\\
        v_p(\alpha_2)&=\cdots=v_p(\alpha_4)=\frac{2}{3}n,\\
        v_p(\alpha_5)&=\cdots=v_p(\alpha_{10})=\frac{1}{2}n, \\
        v_p(\alpha_{11})&=\cdots=v_p(\alpha_{13})=\frac{1}{3}n,\\
        v_p(\alpha_{14})&=0.
    \end{align*}
    Thus Theorem \ref{REF} holds if and only if
    $$f(t)\text{ has no root of valuation }\frac{1}{3}n,\frac{1}{2}n\text{ or }\frac{2}{3}n\text{ and has exactly }2\text{ degree }3\text{ irreducible factors in }\mathbb{Q}_p[t].$$
    
    \item [\textbf{ 1.3}] Suppose that $(1,6n+v_p(a_1))=(1,6n)$ and $(5,2n+v_p(a_5))=(5,3n)$ are vertices. Then $v_p(a_5)=n$ and
    $$\begin{cases}
        5n+v_p(a_2)\geq \frac{21}{4}n, \\
        4n+v_p(a_3)\geq \frac{9}{2}n, \\
        3n+v_p(a_4)\geq \frac{15}{4}n, \\
        n+v_p(a_6)\geq \frac{5}{2}n,\\
        v_p(a_7)\geq 2n.
    \end{cases}\Rightarrow \begin{cases}
        v_p(a_2)\geq \frac{1}{4}n, \\
        v_p(a_3)\geq \frac{1}{2}n, \\
        v_p(a_4)\geq \frac{3}{4}n, \\
        v_p(a_6)\geq \frac{3}{2}n ,\\
        v_p(a_7)\geq 2n.
    \end{cases}$$
    In this case, we have
    $$t-\alpha_1,\:(t-\alpha_2)\cdots(t-\alpha_5),\:(t-\alpha_6)\cdots(t-\alpha_{9}),\:(t-\alpha_{10})\cdots(t-\alpha_{13}),\:t-\alpha_{14}\in\mathbb{Q}_p[t]$$
    with
    \begin{align*}
        v_p(\alpha_1)&=n,\\
        v_p(\alpha_2)&=\cdots=v_p(\alpha_5)=\frac{3}{4}n,\\
        v_p(\alpha_6)&=\cdots=v_p(\alpha_{9})=\frac{1}{2}n, \\
        v_p(\alpha_{10})&=\cdots=v_p(\alpha_{13})=\frac{1}{4}n,\\
        v_p(\alpha_{14})&=0.
    \end{align*}
    Thus Theorem \ref{REF} holds if and only if
    $$f(t)\text{ has no root of valuation }\frac{1}{4}n,\frac{1}{2}n\text{ or }\frac{3}{4}n\text{ and has at most }2\text{ degree }2\text{ irreducible factors in }\mathbb{Q}_p[t].$$

 \item [\textbf{ 1.4}] Suppose that $(1,6n+v_p(a_1))=(1,6n)$ and $(6,n+v_p(a_6))=(6,3n)$ are vertices. Then $v_p(a_6)=2n$ and
    $$\begin{cases}
        5n+v_p(a_2)\geq \frac{27}{5}n, \\
        4n+v_p(a_3)\geq \frac{24}{5}n, \\
        3n+v_p(a_4)\geq \frac{21}{5}n, \\
        2n+v_p(a_5)\geq \frac{18}{5}n,\\
        v_p(a_7)\geq \frac{5}{2}n,
    \end{cases}\Rightarrow \begin{cases}
        v_p(a_2)\geq \frac{2}{5}n, \\
        v_p(a_3)\geq \frac{4}{5}n, \\
        v_p(a_4)\geq \frac{6}{5}n, \\
        v_p(a_5)\geq \frac{8}{5}n, \\
        v_p(a_7)\geq \frac{5}{2}n.
    \end{cases}$$
    In this case, we have
    $$t-\alpha_1,\:(t-\alpha_2)\cdots(t-\alpha_6),\:(t-\alpha_7)(t-\alpha_8),\:(t-\alpha_9)\cdots(t-\alpha_{13}),\:t-\alpha_{14}\in\mathbb{Q}_p[t]$$
    with
    \begin{align*}
        v_p(\alpha_1)&=n,\\
        v_p(\alpha_2)&=\cdots=v_p(\alpha_6)=\frac{3}{5}n,\\
        v_p(\alpha_7)&=v_p(\alpha_{8})=\frac{1}{2}n, \\
        v_p(\alpha_{9})&=\cdots=v_p(\alpha_{13})=\frac{2}{5}n,\\
        v_p(\alpha_{14})&=0.
    \end{align*}
    Thus Theorem \ref{REF} holds if and only if
    $$f(t)\text{ has no root of valuation }\frac{2}{5}n,\frac{1}{2}n\text{ or }\frac{3}{5}n\text{ nor degree }3\text{ irreducible factors in }\mathbb{Q}_p[t].$$

    \item [\textbf{ 1.5}] Suppose that $(1,6n+v_p(a_1))=(1,6n)$ and $(6,n+v_p(a_6))=(6,2n)$ are vertices. Then $v_p(a_6)=n$ and
    $$\begin{cases}
        5n+v_p(a_2)\geq \frac{26}{5}n, \\
        4n+v_p(a_3)\geq \frac{22}{5}n, \\
        3n+v_p(a_4)\geq \frac{18}{5}n, \\
        2n+v_p(a_5)\geq \frac{14}{5}n,\\
        v_p(a_7)\geq \frac{3}{2}n,
    \end{cases}\Rightarrow \begin{cases}
        v_p(a_2)\geq \frac{1}{5}n, \\
        v_p(a_3)\geq \frac{2}{5}n, \\
        v_p(a_4)\geq \frac{3}{5}n, \\
        v_p(a_5)\geq \frac{4}{5}n, \\
        v_p(a_7)\geq \frac{3}{2}n.
    \end{cases}$$
    In this case, we have
    $$t-\alpha_1,\:(t-\alpha_2)\cdots(t-\alpha_6),\:(t-\alpha_7)(t-\alpha_8),\:(t-\alpha_9)\cdots(t-\alpha_{13}),\:t-\alpha_{14}\in\mathbb{Q}_p[t]$$
    with
    \begin{align*}
        v_p(\alpha_1)&=n,\\
        v_p(\alpha_2)&=\cdots=v_p(\alpha_6)=\frac{4}{5}n,\\
        v_p(\alpha_7)&=v_p(\alpha_{8})=\frac{1}{2}n, \\
        v_p(\alpha_{9})&=\cdots=v_p(\alpha_{13})=\frac{1}{5}n,\\
        v_p(\alpha_{14})&=0.
    \end{align*}
    Thus Theorem \ref{REF} holds if and only if
    $$f(t)\text{ has no root of valuation }\frac{1}{5}n,\frac{1}{2}n\text{ or }\frac{4}{5}n\text{ nor degree }3\text{ irreducible factors in }\mathbb{Q}_p[t].$$
 \item [\textbf{ 1.6}] Suppose that $(1,6n+v_p(a_1))=(1,6n)$ and $(7,v_p(a_7))=(7,2n)$ are vertices. Then $v_p(a_7)=2n$ and
    $$\begin{cases}
        5n+v_p(a_2)\geq \frac{16}{3}n, \\
        4n+v_p(a_3)\geq \frac{14}{3}n, \\
        3n+v_p(a_4)\geq 4n, \\
        2n+v_p(a_5)\geq \frac{10}{3}n,\\
        n+v_p(a_6)\geq \frac{8}{3}n,
    \end{cases}\Rightarrow \begin{cases}
        v_p(a_2)\geq \frac{1}{3}n, \\
        v_p(a_3)\geq \frac{2}{3}n, \\
        v_p(a_4)\geq n ,\\
        v_p(a_5)\geq \frac{4}{3}n, \\
        v_p(a_6)\geq \frac{5}{3}n.
    \end{cases}$$
    In this case, we have
    $$t-\alpha_1,\:(t-\alpha_2)\cdots(t-\alpha_7),\:(t-\alpha_8)\cdots(t-\alpha_{13}),\:t-\alpha_{14}\in\mathbb{Q}_p[t]$$
    with
    \begin{align*}
        v_p(\alpha_1)&=n,\\
        v_p(\alpha_2)&=\cdots=v_p(\alpha_7)=\frac{2}{3}n,\\
        v_p(\alpha_{8})&=\cdots=v_p(\alpha_{13})=\frac{1}{3}n,\\
        v_p(\alpha_{14})&=0.
    \end{align*}
    Thus Theorem \ref{REF} holds if and only if
    $$f(t)\text{ has no root of valuation }\frac{1}{3}n\text{ or }\frac{2}{3}n\text{ nor degree }2\text{ irreducible factors in }\mathbb{Q}_p[t].$$

    \item [\textbf{ 1.7}] Suppose that $(1,6n+v_p(a_1))=(1,6n)$ and $(7,v_p(a_7))=(7,n)$ are vertices. Then $v_p(a_7)=n$ and
    $$\begin{cases}
        5n+v_p(a_2)\geq \frac{31}{6}n, \\
        4n+v_p(a_3)\geq \frac{13}{3}n, \\
        3n+v_p(a_4)\geq \frac{7}{2}n, \\
        2n+v_p(a_5)\geq \frac{8}{3}n,\\
        n+v_p(a_6)\geq \frac{11}{6}n,
    \end{cases}\Rightarrow \begin{cases}
        v_p(a_2)\geq \frac{1}{6}n, \\
        v_p(a_3)\geq \frac{1}{3}n, \\
        v_p(a_4)\geq \frac{1}{2}n, \\
        v_p(a_5)\geq \frac{2}{3}n, \\
        v_p(a_6)\geq \frac{5}{6}n.
    \end{cases}$$
    In this case, we have
    $$t-\alpha_1,\:(t-\alpha_2)\cdots(t-\alpha_7),\:(t-\alpha_8)\cdots(t-\alpha_{13}),\:t-\alpha_{14}\in\mathbb{Q}_p[t]$$
    with
    \begin{align*}
       & v_p(\alpha_1)=n,\\
        &v_p(\alpha_2)=\cdots=v_p(\alpha_7)=\frac{5}{6}n,\\
        &v_p(\alpha_{8})=\cdots=v_p(\alpha_{13})=\frac{1}{6}n,\\
        &v_p(\alpha_{14})=0.
    \end{align*}
Thus Theorem \ref{REF} holds if and only if
    $$f(t)\text{ has no root of valuation $\frac{1}{6}n$ or $\frac{5}{6}n$ nor irreducible factors of degree $2$ or $3$ in }\mathbb{Q}_p[t].$$ 
\end{enumerate}
We do a similar analysis whenever one of $(2, 5n+v_p(a_2)), (3, 4n+v_p(a_3)), (4, 3n+v_p(a_4)), (5, 2n+v_p(a_5)), (6, n+v_p(a_6)), (7, v_p(a_7))$ is a starting vertex, and we get the Newton Polygons depicted in Figure~\ref{fig.all}. The information on $v_p(a_i)$ that one reads from these Newton Polygons together with Theorem~\ref{REF} gives the conclusion.
\end{proof}

\begin{remark}
    We remark that the same method would classify all Weil polynomials coming from characteristic polynomials of Frobenius for a simple abelian variety $A$ of dimension $p\geq 5$, for $p$ a prime (see Corollary~\ref{cor:prime}). 
\end{remark}
\begin{figure}
    \centering
    \includegraphics[width=0.4\textwidth]{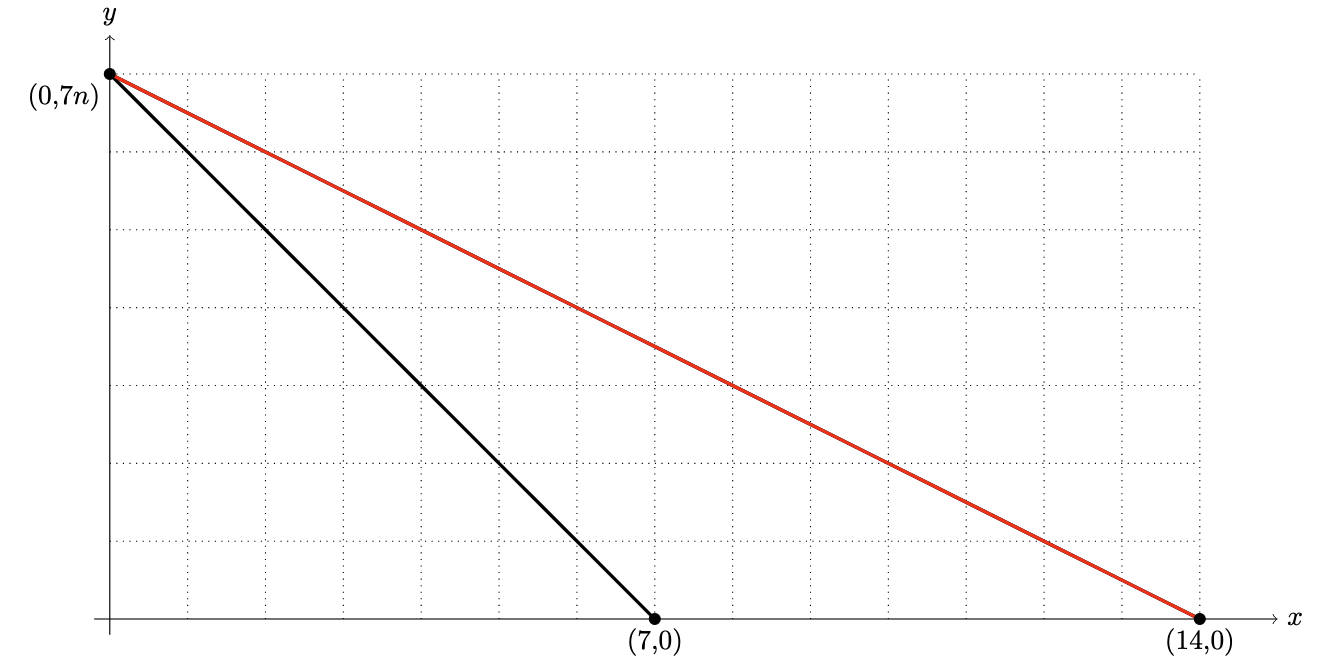}
    \includegraphics[width=0.4\textwidth]{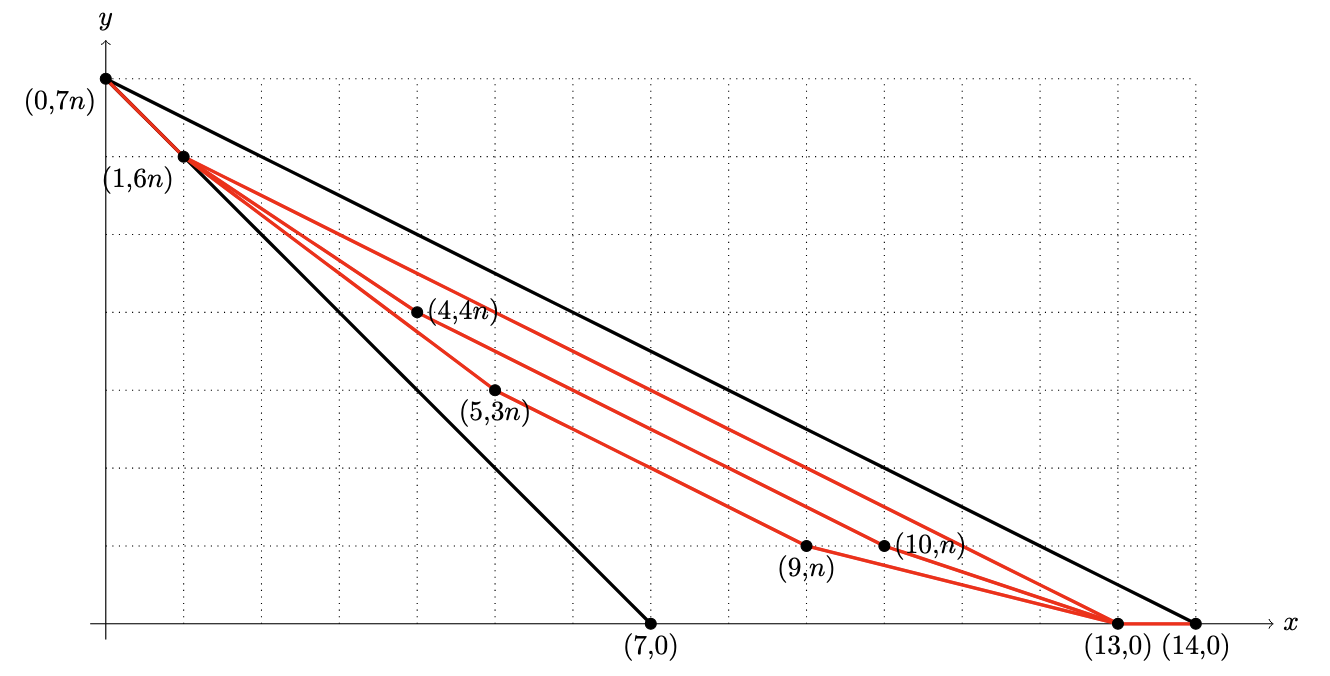}
    \includegraphics[width=0.4\textwidth]{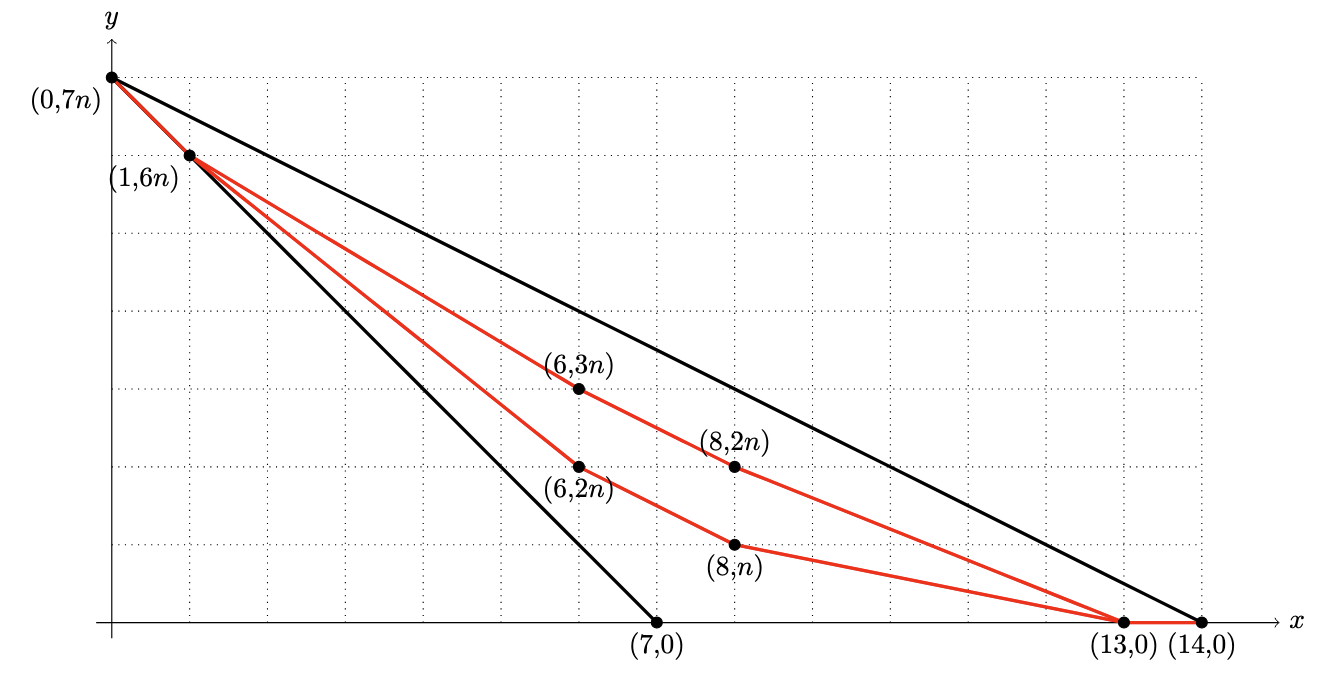}
    \includegraphics[width=0.4\textwidth]{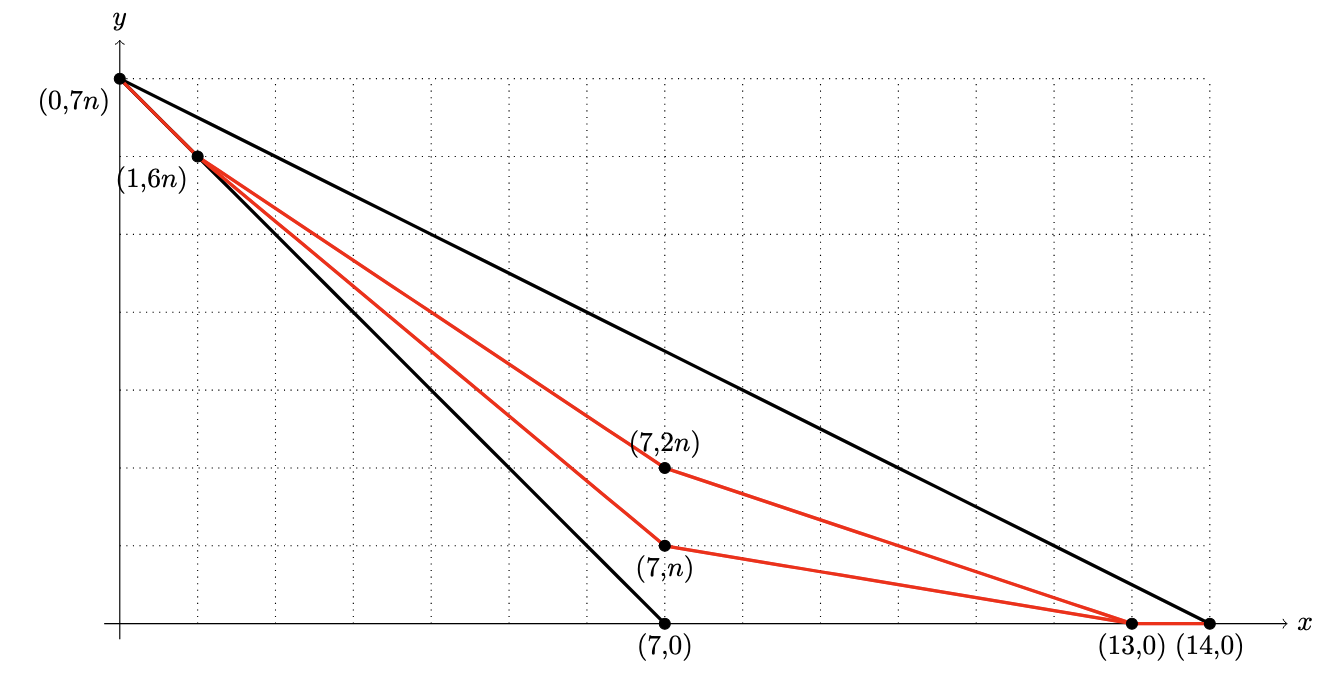}
    \includegraphics[width=0.4\textwidth]{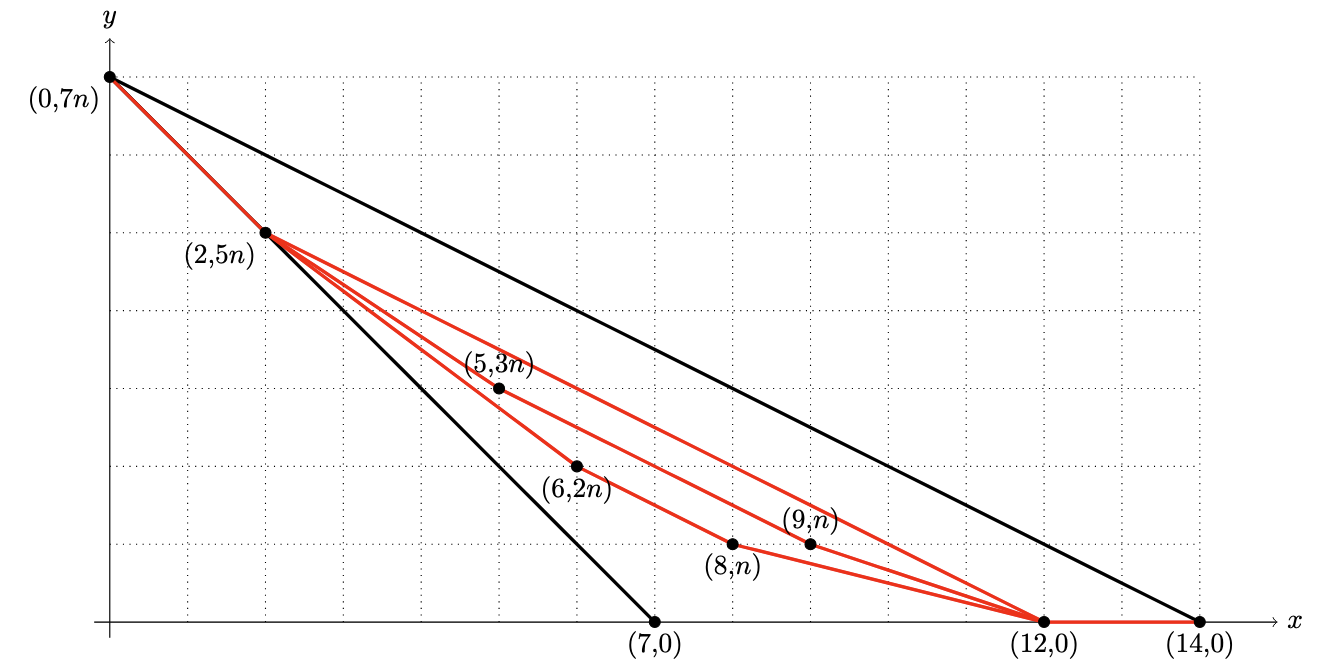}
    \includegraphics[width=0.4\textwidth]{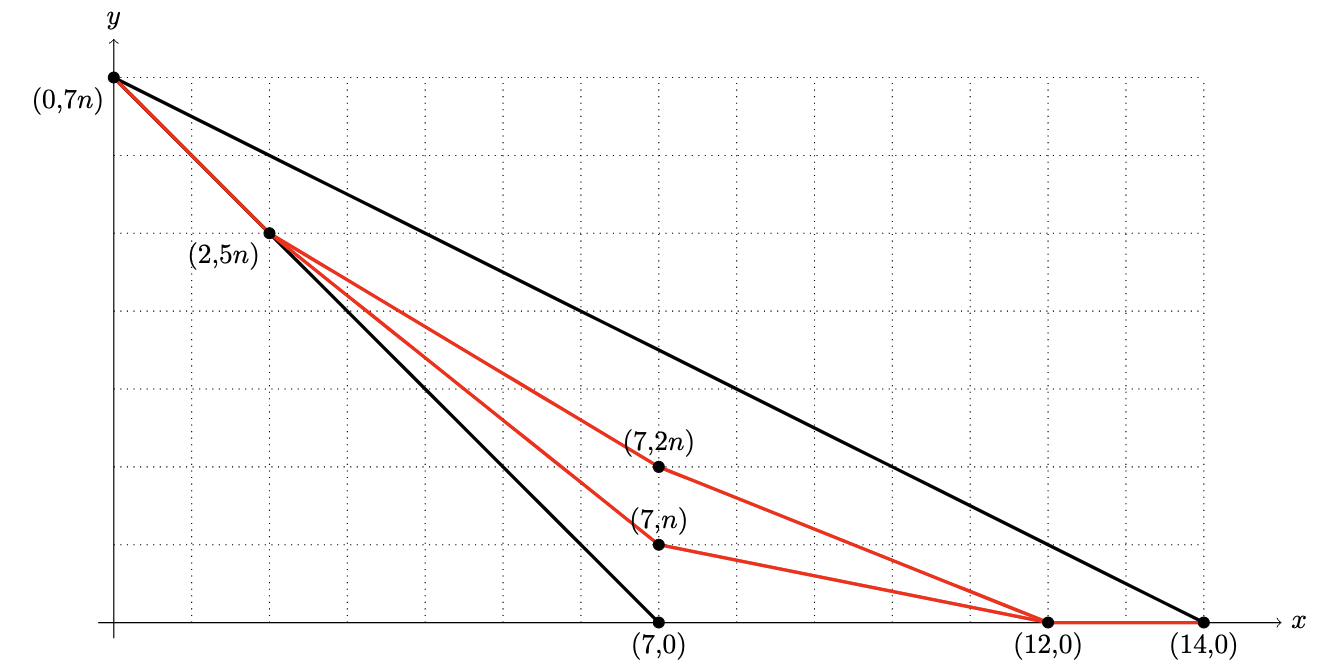}
    \includegraphics[width=0.4\textwidth]{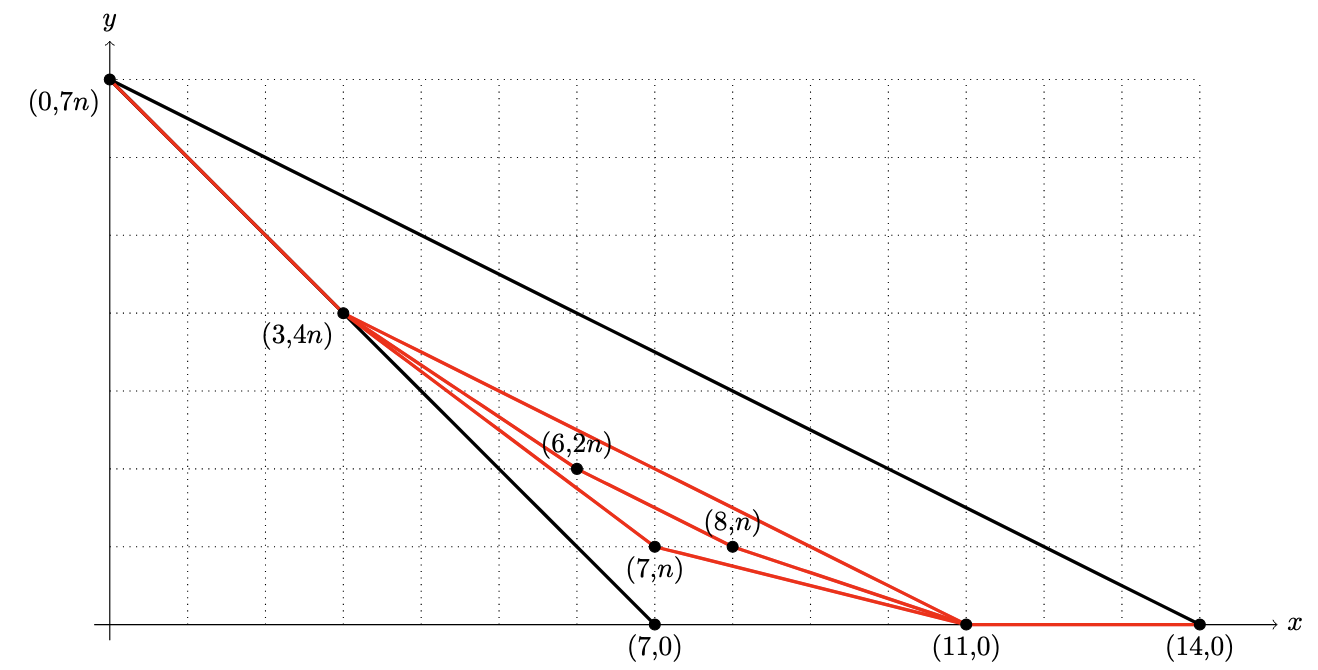}
    \includegraphics[width=0.4\textwidth]{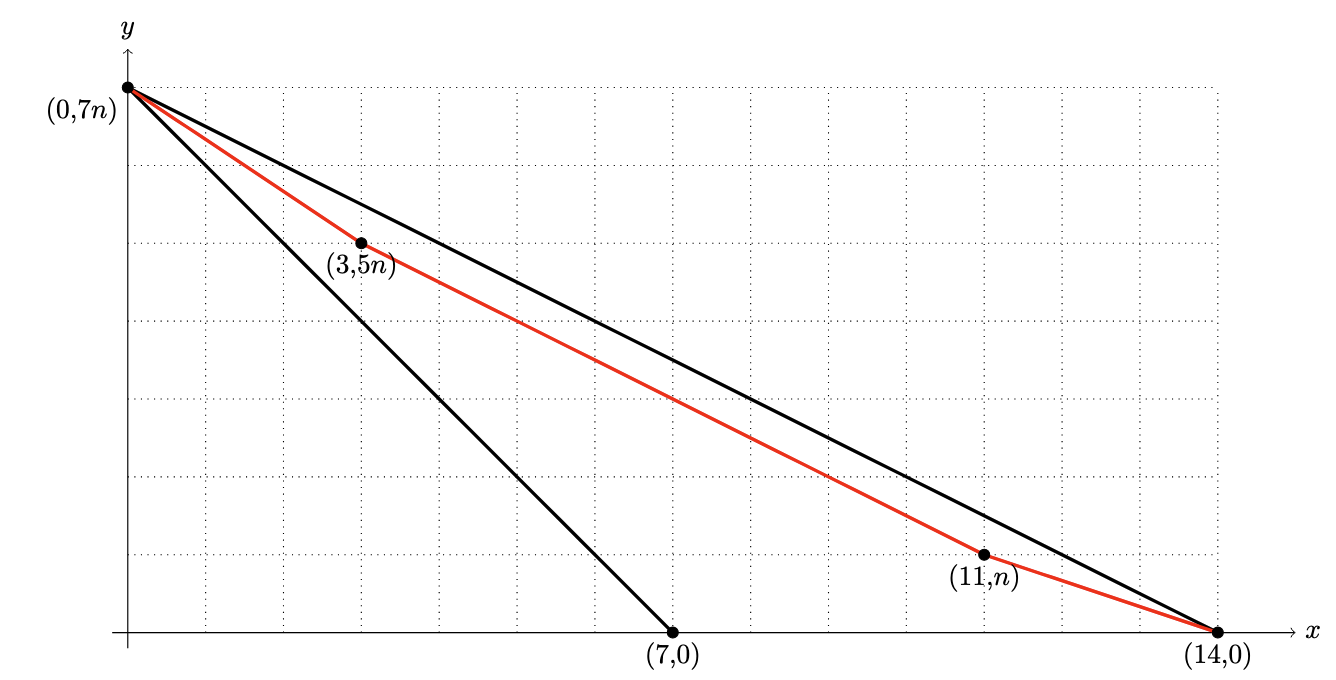}
    \includegraphics[width=0.4\textwidth]{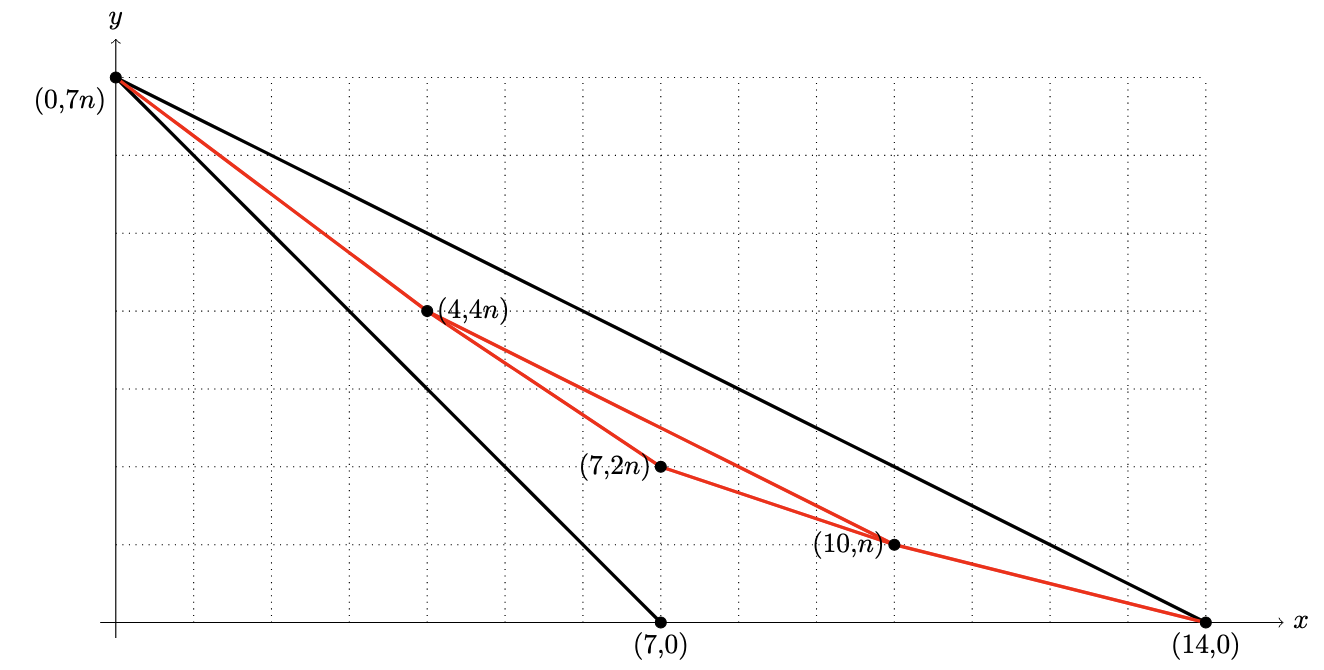}
    \includegraphics[width=0.4\textwidth]{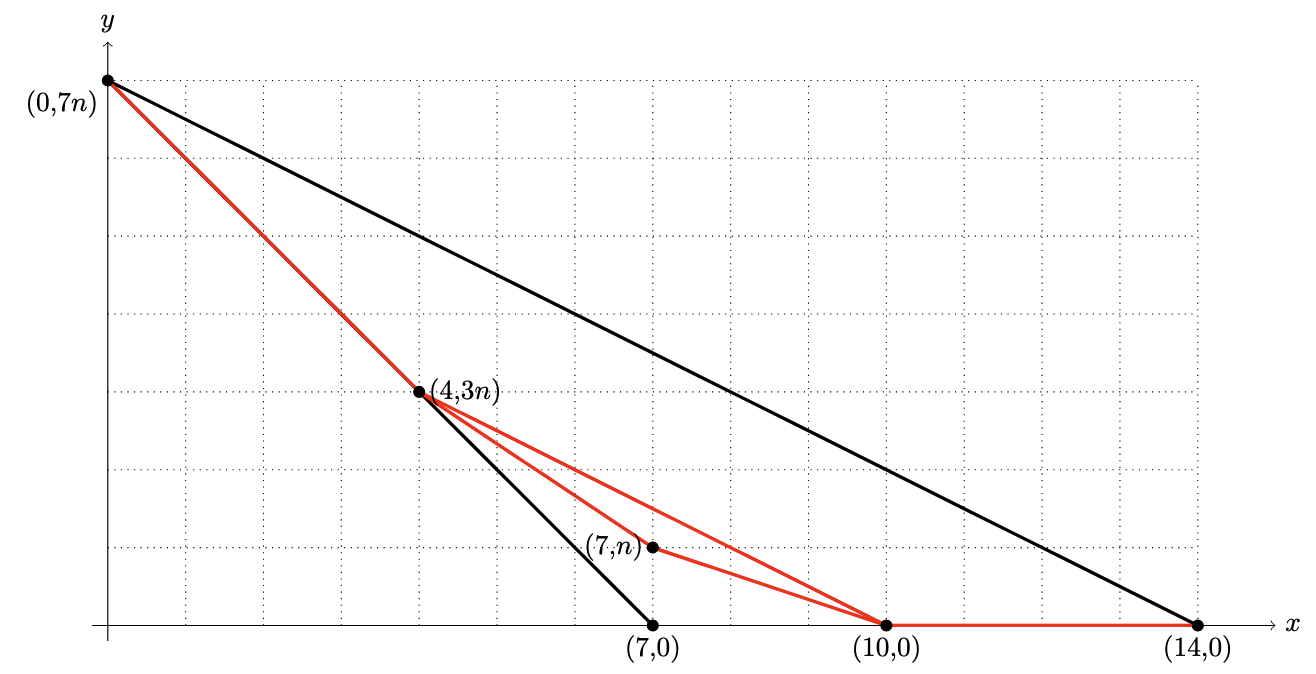}
    \includegraphics[width=0.4\textwidth]{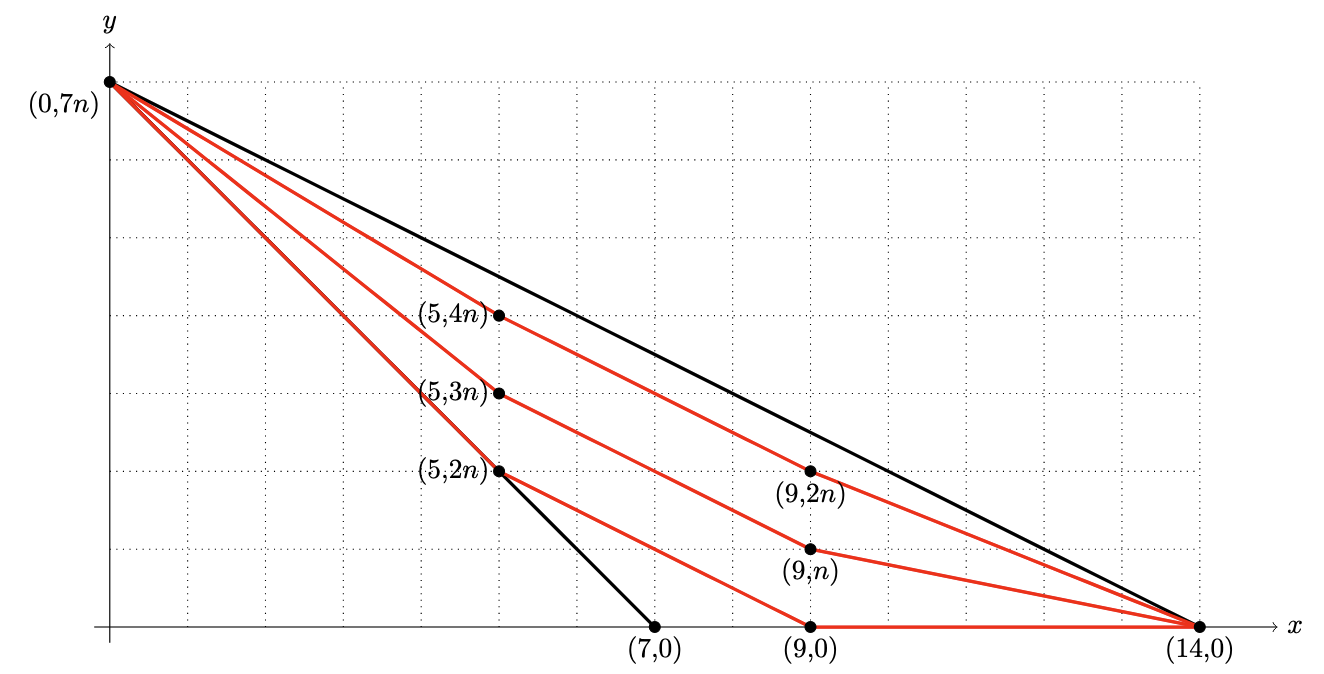}
    \includegraphics[width=0.4\textwidth]{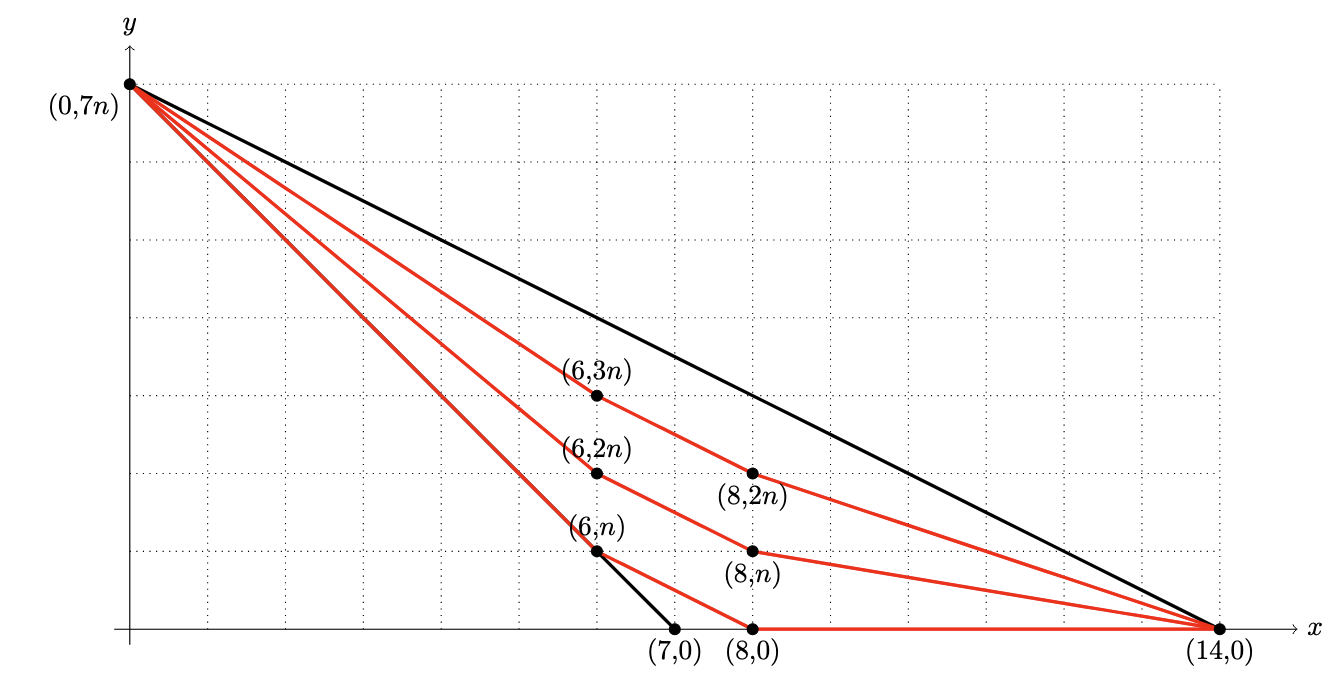}
    \includegraphics[width=0.4\textwidth]{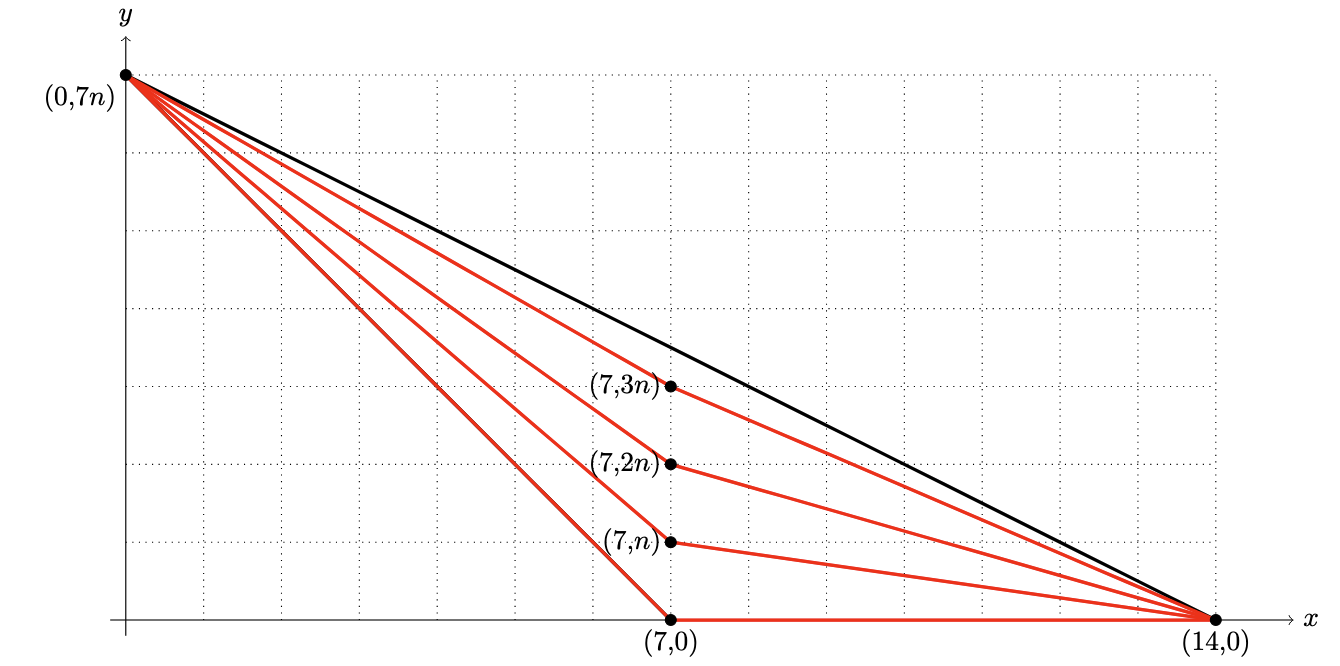}
    \caption{Newton Polygons ordered from left to right, top to bottom, by the starting vertex.}
    \label{fig:cycledata}
\end{figure}\label{fig.all}

\newpage
\bibliographystyle{amsalpha}
    \bibliography{bibliography} 

\end{document}

%% file: macro.tex
\newtheorem{Def}{Definition}[subsection]
\newtheorem{thm}[Def]{Theorem}

\newtheorem{eg}[Def]{Example}
\newtheorem{cor}[Def]{Corollary}
\newtheorem{rmk}[Def]{\normalfont\textit{Remark}}
\newtheorem{prop}[Def]{Proposition}
\newtheorem{propdef}[Def]{Proposition-Definition}
\newtheorem{lem}[Def]{Lemma}

\newtheorem{asp}[Def]{Assumption}
\newtheorem{app}[Def]{Application}
\newtheorem{prob}[Def]{Problem}

\newtheorem{cjt}[Def]{Conjecture}

\newtheorem{rev}{Review}
\newenvironment{pf}{{\noindent\it Proof.}\quad}{\hfill $\qed$\par}

\newcommand{\rad}[1]{\operatorname{rad}(#1)}
\newcommand{\supp}[1]{\operatorname{supp}(#1)}
\newcommand{\spec}[1]{\operatorname{Spec}(#1)}
\newcommand{\Spec}{\operatorname{Spec}}
\newcommand{\spl}{\textnormal{spec}}
\newcommand{\mspec}[1]{\operatorname{MaxSpec}(#1)}

\newcommand{\z}[1]{\mathbb{Z}/#1\mathbb{Z}}
\newcommand{\dis}{\displaystyle}
\newcommand{\mor}[2]{\operatorname{Hom}({#1},{#2})}
\newcommand{\obj}[1]{\operatorname{obj}(#1)}
\newcommand{\homm}[3]{\operatorname{Hom}_{#1}({#2},{#3})}

\newcommand{\udl}[1]{\underline{#1}}
\newcommand{\sat}{\textnormal{sat}}

\newcommand{\fra}[1]{\operatorname{Frac}(#1)}
\newcommand{\im}{\operatorname{im}}
\newcommand{\conj}[1]{\overline{#1}}
\newcommand{\cok}{\operatorname{coker}}
\newcommand{\img}{\operatorname{Im}}
\newcommand{\coim}{\operatorname{coim}}

\newcommand{\inlim}{\varprojlim}
\newcommand{\dlim}{\varinjlim}

\newcommand{\eps}{\varepsilon}

\renewcommand{\AA}{\mathbb{A}}
\newcommand{\CC}{\mathbb{C}}
\newcommand{\EE}{\mathbb{E}}
\newcommand{\FF}{\mathbb{F}}
\newcommand{\GG}{\mathbb{G}}
\newcommand{\HH}{\mathbb{H}}
\newcommand{\NN}{\mathbb{N}}
\newcommand{\QQ}{\mathbb{Q}}
\newcommand{\RR}{\mathbb{R}}
\newcommand{\KK}{\mathbb{K}}
\newcommand{\LL}{\mathbb{L}}
\newcommand{\ZZ}{\mathbb{Z}}
\newcommand{\bfm}{\mathbf{m}}
\newcommand{\mcA}{\mathcal{A}}
\newcommand{\mcB}{\mathcal{B}}
\newcommand{\mcC}{\mathcal{C}}
\newcommand{\mcD}{\mathcal{D}}
\newcommand{\mcF}{\mathcal{F}}
\newcommand{\mcG}{\mathcal{G}}
\newcommand{\mcH}{\mathcal{H}}
\newcommand{\mcL}{\mathcal{L}}
\newcommand{\mcI}{\mathcal{I}}
\newcommand{\mcJ}{\mathcal{J}}
\newcommand{\mcM}{\mathcal{M}}
\newcommand{\mcN}{\mathcal{N}}
\newcommand{\mcO}{\mathcal{O}}
\newcommand{\mcP}{\mathcal{P}}
\newcommand{\mcU}{\mathcal{U}}
\newcommand{\mcS}{\mathcal{S}}
\newcommand{\mcQ}{\mathcal{Q}}
\newcommand{\mcZ}{\mathcal{Z}}
\newcommand{\mfa}{\mathfrak{a}}
\newcommand{\mfA}{\mathfrak{A}}
\newcommand{\mfb}{\mathfrak{b}}
\newcommand{\mfB}{\mathfrak{B}}
\newcommand{\mfC}{\mathfrak{C}}
\newcommand{\mfD}{\mathfrak{D}}
\newcommand{\mfF}{\mathfrak{F}}
\newcommand{\mff}{\mathfrak{f}}
\newcommand{\mfj}{\mathfrak{j}}
\newcommand{\mfI}{\mathfrak{I}}
\newcommand{\mfM}{\mathfrak{M}}
\newcommand{\mfm}{\mathfrak{m}}
\newcommand{\mfN}{\mathfrak{N}}
\newcommand{\mfn}{\mathfrak{n}}
\newcommand{\mfo}{\mathfrak{o}}
\newcommand{\mfO}{\mathfrak{O}}
\newcommand{\mfP}{\mathfrak{P}}
\newcommand{\mfQ}{\mathfrak{Q}}
\newcommand{\mfR}{\mathfrak{R}}
\newcommand{\mfS}{\mathfrak{S}}
\newcommand{\mfT}{\mathfrak{T}}
\newcommand{\mfU}{\mathfrak{U}}
\newcommand{\mfV}{\mathfrak{V}}
\newcommand{\mfW}{\mathfrak{W}}
\newcommand{\mfX}{\mathfrak{X}}
\newcommand{\mfY}{\mathfrak{Y}}
\newcommand{\mfZ}{\mathfrak{Z}}
\newcommand{\mfp}{\mathfrak{p}}
\newcommand{\mfq}{\mathfrak{q}}
\newcommand{\mfz}{\mathfrak{z}}
\newcommand{\bfa}{\mathbf{A}}
\newcommand{\bfp}{\mathbf{P}}
\newcommand{\AGL}{\mathbb{A}\GL}
\newcommand{\Qbar}{\overline{\QQ}}
\newcommand{\dmn}{\trianglerighteq}

\newcommand{\abs}[1]{\left|#1\right|}
\newcommand{\pphi}{\varphi}
\newcommand{\upto}[1]{\overset{#1}{\to}}
\newcommand{\res}[2]{\textnormal{res}^{#1}_{#2}}
\newcommand{\op}{\textnormal{op}}
\newcommand{\nat}[2]{\textnormal{Nat}(#1,#2)}
\newcommand{\comr}{\textnormal{\textbf{ComRings}}}
\newcommand{\Mod}{\textnormal{\textbf{Mod}}}
\newcommand{\lmod}{\sideset{_R}{}{\mathop{\Mod}}}
\newcommand{\rmod}{\sideset{}{_R}{\mathop{\Mod}}}
\newcommand{\mmod}[1]{\sideset{_{#1}}{}{\mathop{\Mod}}}
\newcommand{\set}{\textnormal{\textbf{Sets}}}
\newcommand{\grp}{\textnormal{\textbf{Grp}}}
\newcommand{\Ab}{\textnormal{{\textbf{Ab}}}}
\newcommand{\Top}{\textnormal{\textbf{Top}}}
\newcommand{\en}[2]{\textnormal{End}_{#1}(#2)}
\newcommand{\hgt}[1]{\textnormal{ht}(#1)}
\newcommand{\gr}[2]{\textnormal{gr}_{#1}(#2)}
\newcommand{\grd}[3]{\textnormal{gr}_{#1}^{#2}(#3)}
\newcommand{\br}{\operatorname{Br}}
\newcommand{\ab}{\operatorname{ab}}
\newcommand{\gal}{\operatorname{Gal}}
\newcommand{\rig}{\textnormal{rig}}

\newcommand{\N}{\operatorname{N}}
\newcommand{\h}{\operatorname{H}}
\newcommand{\disc}[1]{\textnormal{disc}({#1})}
\newcommand{\norm}[1]{|\!|#1|\!|}
\newcommand{\bignorm}[1]{\bigg|\!\bigg|#1\bigg|\!\bigg|}
\renewcommand{\mod}{\operatorname{mod}}
\renewcommand{\Re}{\operatorname{Re}}
\renewcommand{\Im}{\operatorname{Im}}
\newcommand{\Gal}{\operatorname{Gal}}
\newcommand{\cov}{\operatorname{Cov}}
\newcommand{\cat}{\operatorname{Cat}}
\renewcommand{\op}{\operatorname{op}}
\newcommand{\eff}{\textnormal{eff}}
\newcommand\nn             {\nonumber \\}

\newcommand{\tdiv}{\operatorname{\mid\!\mid}}

\newcommand\be            {\begin{equation}}
\newcommand\ee            {\end{equation}}
\newcommand\bea           {\begin{eqnarray}}
\newcommand\eea         {\end{eqnarray}}
\newcommand\bnu          {\begin{enumerate}}
\newcommand\enu          {\end{enumerate}}
\newcommand{\Ad}{\operatorname{Ad}}

\newcommand\id            {\mathrm{id}}
\newcommand\ob          {\operatorname{Ob}}
\renewcommand\hom         {\operatorname{Hom}}
\newcommand\ev          {\mathrm{ev}}
\newcommand\coev      {\mathrm{coev}}
\newcommand\edo    {\mathrm{End}}
\newcommand\funend {\EuScript{E}\mathrm{nd}}
\newcommand\aut      {\mathrm{Aut}}
\newcommand\inn      {\mathrm{Inn}}
\newcommand\out      {\mathrm{Out}}
\newcommand\Aut      {\mathcal{A}\mathrm{ut}}
\newcommand\hilb   {\mathrm{Hilb}}
\newcommand\vect    {\mathrm{Vect}}
\newcommand\Mat  {\EuScript{M}\mathrm{at}}
\renewcommand{\div}{\operatorname{div}}
\newcommand\rep     {\mathrm{Rep}}
\newcommand\fun     {\mathrm{Fun}}
\newcommand\Fun    {\EuScript{F}\mathrm{un}}
\newcommand\LMod  {\mathrm{LMod}}
\newcommand\RMod  {\mathrm{RMod}}
\newcommand\BMod {\mathrm{BMod}}
\newcommand\bk       {\mathbb{k}}
\newcommand\forget  {\mathbf{f}}

\newcommand{\sep}{\textnormal{sep}}

\newcommand\alg     {\EuScript{A}\mathrm{lg}}
\newcommand\cTop {\mathrm{Top}}
\newcommand\mfd    {\EuScript{M}\mathrm{fd}}
\newcommand\Set    {\EuScript{S}\mathrm{et}}
\newcommand\sSet    {\mathbf{sSet}}
\newcommand\Topo    {\EuScript{T}\mathrm{op}}
\newcommand\Ring  {\EuScript{R}\mathrm{ing}}
\newcommand\abel {\EuScript{A}\mathrm{bel}}
\newcommand\cring {\EuScript{C}\mathrm{ring}}

\newcommand\one    {\mathbf{1}}
\newcommand\coker  {\mathrm{Coker}}
\newcommand\image     {\mathrm{Im}}

\newcommand{\auto}[1]{\operatorname{Aut}(#1)}
\newcommand{\inv}[1]{\operatorname{Inv}(#1)}
\newcommand{\ch}[1]{\operatorname{char}(#1)}
\newcommand{\colim} {\varinjlim}
\newcommand{\Lim}  {\varprojlim}
\newcommand{\Lan}  {\mathrm{Lan}}
\newcommand{\Ran}  {\mathrm{Ran}}
\newcommand{\Lie}{\operatorname{Lie}}
\newcommand{\BB}{\mathbb{B}}
\newcommand{\Max}{\operatorname{Max}}
\renewcommand{\sp}{\operatorname{Sp}}
\newcommand{\End}{\operatorname{End}}
\renewcommand{\Mat}{\operatorname{Mat}}

\newcommand{\mot}{\textnormal{-}\mathfrak{Mot}}
\newcommand{\andm}{\textnormal{-}\mathfrak{AndMod}}
\newcommand{\abm}{\textnormal{-}\mathfrak{AbMod}}
\newcommand{\rk}{\operatorname{rk}}
\newcommand{\Frac}{\operatorname{Frac}}

\newcommand\CA           {\EuScript{A}}
\newcommand\CB           {\EuScript{B}}
\newcommand\CCC           {\EuScript{C}}
\newcommand\CDD           {\EuScript{D}}
\newcommand\CE          {\EuScript{E}}
\newcommand\CF          {\EuScript{F}}
\newcommand\CG         {\EuScript{G}}
\newcommand\CH         {\EuScript{H}}
\newcommand\CI           {\EuScript{I}}
\newcommand\CJ           {\EuScript{J}}
\newcommand\CK         {\EuScript{K}}
\newcommand\CL          {\EuScript{L}}
\newcommand\CM          {\EuScript{M}}
\newcommand\CN         {\EuScript{N}}
\newcommand\CO         {\EuScript{O}}
\newcommand\CP         {\EuScript{P}}
\newcommand\CQ         {\EuScript{Q}}
\newcommand\CR         {\EuScript{R}}
\newcommand\CS         {\EuScript{S}}
\newcommand\CT         {\EuScript{T}}
\newcommand\CU        {\EuScript{U}}
\newcommand\CV        {\EuScript{V}}
\newcommand\CW        {\EuScript{W}}
\newcommand\CX         {\EuScript{X}}
\newcommand\CY         {\EuScript{Y}}
\newcommand\CZ         {\EuScript{Z}}
\newcommand{\plim}{\varprojlim}

\newcommand{\fA}{\mathfrak{A}}
\newcommand{\fB}{\mathfrak{B}}
\newcommand{\sF}{\mathscr{F}}
\newcommand{\sO}{\mathscr{O}}
\newcommand{\sG}{\mathscr{G}}
\newcommand{\sH}{\mathscr{H}}
\newcommand{\sI}{\mathscr{I}}
\newcommand{\sN}{\mathscr{N}}
\newcommand{\sR}{\mathscr{R}}
\newcommand{\sP}{\mathscr{P}}
\newcommand{\uparr}[1]{\stackrel{\to}{#1}}

\newcommand{\SL}{\operatorname{SL}}
\newcommand{\GL}{\operatorname{GL}}
\newcommand{\sgn}{\operatorname{sgn}}
\newcommand{\ksym}[2]{\left(\frac{#1}{#2}\right)}